\documentclass[10pt, reqno]{amsart}
\usepackage{amscd} 
\usepackage{amsfonts} 
\usepackage{amssymb} 
\usepackage{latexsym} 
\usepackage[arrow, matrix, curve]{xy}

\input diagrams

\newcommand{\ncm}{\newcommand}

\newtheorem{theorem}{Theorem}[section]
\newtheorem{prop}[theorem]{Proposition}
\newtheorem{problem}[theorem]{Problem}
\newtheorem{lemma}[theorem]{Lemma}
\newtheorem{cor}[theorem]{Corollary}
\newtheorem{lem&def}[theorem]{Lemma \& Definition}
\newtheorem{definition}[theorem]{Definition}
\newtheorem{example}[theorem]{Example}
\newtheorem{remark}[theorem]{Remark}

\def\M{\mathcal{M}}

\def\C{\mathbb{C}\,} 
\def\Z{\mathbb{Z}\,} 
 
\def\N{\mathbb{N}\,}

\ncm{\Ann}{\mbox{\rm Ann}\,}
\ncm{\End}{\mbox{\rm End}\,}

\def\Hom{\mbox{\rm Hom}\,}
\def\Indec{\mbox{\rm Indec}\,}

\def\|{\, | \,}
\def\h{\, \sim \,}

\def\id{\mbox{\rm id}}

\def\into{\hookrightarrow}
\def\to{\rightarrow}
\def\Ind{\mbox{\rm Ind}}
\def\Coind{\mbox{\rm CoInd}}
\def\Res{\mbox{\rm Res}}

\def\o{\otimes}    
\def\bra{\langle}
\def\ket{\rangle}

\ncm{\rarr}[1]{\stackrel{#1}{\longrightarrow}}
\ncm{\larr}[1]{\stackrel{#1}{\longleftarrow}}
\def\cop{\Delta}

\def\eps{\varepsilon}

\def\-2{_{(-2)}}
\def\-1{_{(-1)}}
\def\0{_{(0)}}
\def\1{_{(1)}}
\def\2{_{(2)}}
\def\3{_{(3)}}
\def\n{_{(n)}}

\def\du1{\hat 1}
\def\lact{\triangleright}
\def\ract{\triangleleft}

\begin{document}
\title[Algebra  depth in tensor categories]{Algebra  depth in \\  tensor categories}
\author{Lars Kadison \\  \\ \\  \textit{In memory of Daniel Kastler}} 
\address{Departamento de Matematica \\ Faculdade de Ci\^encias da Universidade do Porto \\ 
Rua Campo Alegre, 687 \\ 4169-007 Porto, Portugal} 
\email{lkadison@fc.up.pt } 
\date{} 

\begin{abstract}
 Study of the quotient module of a finite-dimensional Hopf subalgebra pair in order to compute its depth yields a relative Maschke Theorem, in which semisimple extension is characterized as being separable, and is therefore an ordinary Frobenius extension.  We study the core Hopf ideal of a Hopf subalgebra, noting that the length of the annihilator chain of tensor powers of the quotient module is linearly related to the depth, if the Hopf algebra is semisimple. A tensor categorical definition of depth is introduced, and a summary from this new point of view of previous results are included. It is shown in a last section that the depth, Bratteli diagram and relative cyclic homology of  algebra extensions are Morita invariants.  
\end{abstract} 
\maketitle

\section{Introduction and Preliminaries}

Sometimes it is useful to classify numbers with the same prime factors together.  Similarly, it is useful to
classify together finite-dimensional modules over a finite-dimensional algebra  with isomorphic indecomposable summands -  two such modules, which have the same indecomposables but perhaps with different nonzero multiplicities, are said to be similar.  Since an abelian category has direct sum
$\oplus$ that work as usual, similarity of two objects $X,Y$, denoted by $X \sim Y$, is defined by $X \oplus * \cong n \cdot Y$, i.e., ``$X$ divides a multiple of $Y$,''
and $Y \oplus * \cong m \cdot X$ (or briefly $Y \| m \cdot X$) for some multiplicities $m,n \in \N$.  In the presence of a uniqueness theorem for indecomposables that includes $X,Y$, they share isomorphic indecomposable summands. Also, the endomorphism rings of $X$ and $Y$ are Morita equivalent in a particularly transparent way \cite{AF, H}.  For example, one may introduce the theory of basic algebras without complications using the regular representation and a similar direct sum of projective indecomposables with constant multiplicity  one.  

A special type of abelian category is a tensor category,  which has a tensor product $\otimes$ satisfying the usual distributive, associative and unital laws up to natural isomorphism.  An algebra $A$ may then be defined in terms of
multiplication $A \otimes A \rightarrow A$ as usual. Define the minimum \textit{depth} of $A$ to be the least $2n+1 = 1,3,5,\ldots$ such that $A^{\otimes (n)} = A\otimes \cdots \otimes A$ ($n$ times $A$) is similar to $A^{\otimes (n+1)}$, which simplifies to 
$A^{\otimes (n+1)} \| q \cdot A^{\otimes (n)}$ for some $q \in \N$, since $A^{\otimes (n)} \|  A^{\otimes (n+1)}$ follows from applying the multiplication and unit.  This definition applied to an algebra $A$ in the category of bimodules over a ring $B$ with tensor $\otimes = \otimes_B$,  recovers the minimum odd depth of the ring extension
$B \rightarrow A$ \cite{BDK}, where it is applied to finite group algebra extensions to recover (together with minimum even depth) subgroup depth \cite{BKK}.  Interesting values of subgroup depth have been computed in \cite{BKK, BDK, D, F, FKR, HHP}, where subgroup depth less than $3$ are normal subgroups \cite{BK, BK2, KS, LKProc, LK2005}.  Several properties of subgroup depth extend to Hopf subalgebra (and left coideal subalgebra) pairs such as a characterization of normality \cite{BK} and unchanged minimum even depth when factoring out the subgroup core \cite{BDK, HKS}.

The main problem in the area is the one formulated in \cite[p.\ 259]{BDK} for a finite-dimensional Hopf subalgebra pair $R \subseteq H$, where $d(R,H)$ denotes the minimum depth.  

\begin{problem}
Is $d(R,H) < \infty$?
\end{problem}
There are examples in subfactor theory by Haagerup of infinite depth, although not anwering the problem.  
We bring up three other equivalent problems below.  

In the opposite tensor category, algebra becomes a notion of coalgebra with the same definition of depth.  In the tensor category of bimodules over $B$, a coalgebra in this sense is a $B$-coring. Applying the definition of depth to the Sweedler coring of a ring extension, one recovers the minimum h-depth of the ring extension as defined in \cite{K2012}. The minimum h-depth of a Hopf subalgebra pair $R \subseteq H$ is shown in \cite{K2014} to be precisely determined by the depth of their quotient module $Q_H = H/R^+H$  in the finite tensor category  of finite-dimensional $H$-modules \cite{EO}. 
In turn the depth of $Q$ is  determined precisely by the length of the descending chain of annihilator ideals of
the tensor powers of  $Q$,  if the Hopf algebra is semisimple, as proven in Theorem~\ref{th-precise}. The quotient module $Q$ has many uses, including the following equivalent reformulation of the problem above, either as an $H$-
or  $R$-module isoclass in the respective representation ring (see  \cite{K2014} or Section~3, the notion below is algebraic element in a ring). 

\begin{problem}
Is $Q$ an algebraic module?
\end{problem}
For example, a finite group algebra extension has quotient module $Q$ equal to a permutation module, which is algebraic \cite[Ch.\ 9]{F}.  The question in general is only interesting for the projective-free summands of $Q$, since projectives form a finite rank ideal in the representation ring \cite{HKY}.  If
either $R$ or $H$ has finite representation type (e.g., is semisimple, Nakayama serial), $Q$ is similarly algebraic.  Example~\ref{ex-eight} computes a finite depth where both Hopf algebras are of infinite representation type.  

In Section~4, we study depth of a non-normal subalgebra in a factorisable Hopf algebra in terms of entwined subalgebras such as a matched pair of Hopf algebras.   In Section~3, we prove a relative Maschke theorem characterizing semisimple extension of finite-dimensional Hopf algebras as a separable extension; as a corollary, these are
ordinary (or untwisted) Frobenius extensions.  We also define and study the core Hopf ideal of a Hopf subalgebra, which extends to Hopf algebras the usual notion of core of a subgroup pair of finite groups. We note that the length of the annihilator chain of tensor powers of the quotient module is linearly related to the depth if the Hopf algebra is semisimple, improving on some results in \cite{HKY}.   In Section~5, we make a categorical study of a Morita equivalence of noncommutative ring extensions. We show that depth and relative cyclic homology of a ring extension are Morita invariants, as is the inclusion matrix of a semisimple complex algebra extension.  

\subsection{Similar modules} Let $A$ be a ring.  Two left $A$-modules, ${}_AN$ and ${}_AM$, are said to be \textit{similar} (\cite{AF}, or H-equivalent \cite{H}) denoted
by ${}_AM \h {}_AN$ if two
conditions are met.  First, 
for some positive integer $r$,  $N$ is isomorphic to a direct summand in the direct sum of $r$ copies of $M$, denoted by  \hspace{1cm} 
${}_AN \oplus * \cong r \cdot {}_AM \  \  \  \Leftrightarrow \   $
\begin{equation}
\label{eq: oplus}
  N \| r \cdot M \  \Leftrightarrow \ \  \exists f_i \in \Hom ({}_AM,{}_AN),  \ g_i \in \Hom ({}_AN, {}_AM): \ \sum_{i=1}^r f_i \circ g_i = \id_N 
\end{equation}
 Second, symmetrically there is $s \in \Z_+$ such that $M \| s \cdot N$. (Say that $M$ and $N$ are \textit{dissimilar} if neither condition 
$M \| s \cdot N$ or $N \| r \cdot M$ holds.)
It is easy to extend this definition of  similarity to similarity of two objects in an abelian category, and to show that it is an equivalence relation.

\begin{example}  
\label{ex-artinian}
\begin{rm}
Suppose $A$ is an artinian ring,  with indecomposable $A$-modules $\{ P_{\alpha} | \alpha \in I  \}$ (representatives from each isomorphism class for some index set $I$).  By Krull-Schmidt finitely generated modules
$M_A$ and $N_A$ have a unique factorization into a direct sum of multiples of finitely many indecomposable module components.  Denote the indecomposable constituents
of $M_A$ by $\Indec(M) = \{ P_{\alpha} \| [P_{\alpha},M] \neq 0 \}$ where $[P_{\alpha},M]$ is the number of factors
in $M$ isomorphic to $P_{\alpha}$.  Note that $M \| q \cdot N$ for some positive $q$ if and only if  $\Indec(M) \subseteq \Indec(N)$. It follows that  $M \h N$ iff $\Indec(M) = \Indec(N)$.  

Suppose $A_A = n_1 P_1 \oplus \cdots \oplus n_rP_r$ is the decomposition of the regular module
into its projective indecomposables.  Let $P_A = P_1 \oplus \cdots \oplus P_r$.  Then $P_A$
and $A_A$ are similar (and call $P$ the \textit{basic} $A$-module in the similarity class of $A$).  Then $A$ and $\End P_A$ are
Morita equivalent.  The algebra $\End P_A$ is of course the basic algebra of $A$. 

Suppose $A$ is a semisimple ring.  Then $P_i = S_i$ are simple modules.  Note that the annihilator ideal $\Ann S_i$ is a maximal ideal in $A$; denote it
by $I_i$.   Note that $\Ann (n_i \cdot S_i) = I_i$, $\Ann ( n_i \cdot S_i \oplus n_j \cdot S_j) =  I_i \cap I_j$, and any ideal $I$ is uniquely $\Ann (S_{i_1} \oplus \cdots \oplus S_{i_s})$ for the $2^r$ integer subsets, $1 \leq i_1 < \cdots < i_s \leq r$.
\end{rm} 
 \end{example}

\begin{prop}
\label{prop-gem}
If two modules are similar, then their annihilator ideals are equal.  Conversely, if $A$ is a semisimple ring, two finitely generated modules with equal annihilator ideals are similar.
\end{prop}
\begin{proof}
Given modules $M$ and $N$, if $M \into N$, then $\Ann N \subseteq \Ann M$.  It follows that $M \| r \cdot N$ implies that $\Ann N \subseteq \Ann M$.  Hence, $M \h N \Rightarrow \Ann M = \Ann N$. 

Suppose $A$ is a semisimple ring; we use the notation in the example.  
If $M$ and $N$ are finitely generated $A$-modules such that $\Ann M = \Ann N$ is the ideal $I$ in $A$, then $I = I_{i_1} \cap \cdots \cap I_{i_s}$ for some integers $1 \leq i_1 < \cdots < i_s \leq r$.  It follows that $S_{i_1} \oplus \cdots \oplus S_{i_s}$ is the basic module in the similarity class of both $M$
and $N$; in particular, $M \h N$.  
\end{proof}

\begin{example}
\begin{rm}
Suppose $R$ is an artinian ring that is not semisimple and with two additional indecomposable modules $I_1, I_2$ that are not projective and not isomorphic.  Then the modules $M = R \oplus I_1$
and $N = R \oplus I_2$ are both faithful generators,
but dissimilar by Krull-Schmidt.  This contradicts the converse of the proposition for more general rings.  
(Without dissimilarity, one additional nonprojective indecomposable would suffice.)
\end{rm}
\end{example}
 
\subsection{Subring depth} 
 Throughout this section, let $A$ be a unital associative ring and $B \subseteq A$ a subring where $1_B =  1_A$; more generally, it suffices to assume $B \to A$ is a unital ring homomorphism,  called a ring extension, although we suppress this option notationally. Note the natural bimodules ${}_BA_B$ obtained by restriction of the natural $A$-$A$-bimodule (briefly $A$-bimodule) $A$, also  to the natural bimodules ${}_BA_A$, ${}_AA_B$ or ${}_BA_B$, which
are referred to with no further ado.  
Let $A^{\otimes_B (n)}$ denote $A \otimes_B \cdots \otimes_B A$ ($n$ times $A$, $n \in \N$), where
$A^{\otimes_B 0} = B$. 
For $n \geq 1$, the $A^{\otimes_B (n)}$ has a natural $A$-bimodule  structure which restricts
to $B$-$A$-, $A$-$B$- and $B$-bimodule structures occuring in the next definition. 
Note that $A^{\otimes_B (n)} \| A^{\otimes_B (n+1)}$ automatically occurs in any case for $n \geq 2$, since $A \rightarrow A \otimes_B A$ given by $a \mapsto a \otimes_B 1$ is a split monomorphism. For $n = 1$ and $A$-bimodules, this is the separability condition on $A \supseteq B$; otherwise, $A \| A\otimes_B A$ as $A$-$B$- or $B$-$A$-bimodules (via the split epi $a \otimes_B a' \mapsto aa'$).   

\begin{definition}
\label{def-depth}
\begin{rm}
The subring $B \subseteq A$ has depth $2n+1 \geq 1$ if as $B$-bimodules $ A^{\otimes_B (n)} \h A^{\otimes_B (n+1)} $.  The subring  $B\subseteq A$ has left (respectively, right) depth $2n \geq 2$ if $ A^{\otimes_B (n)} \h A^{\otimes_B (n+1)} $ as $B$-$A$-bimodules (respectively, $A$-$B$-bimodules). Equivalently, $A \supseteq B$ has depth $2n+1 \geq 1$, or left depth $2n \geq 2$, if
\begin{equation}
\label{eq: dirsum}
A^{\otimes_B (n+1)} \oplus * \cong q \cdot A^{\otimes_B (n)}
\end{equation}
as $B$-$B$-bimodules, or $B$-$A$-bimodules, respectively.  Right depth $2n$ is defined similarly in terms of $A$-$B$-bimodules.   

It is clear that if $B \subseteq A$ has either left or right depth $2n$, it has depth $2n+1$ by restricting the similarity  condition to $B$-bimodules.   If $B \subseteq A$ has depth $2n+1$, it has depth $2n+2$ by
tensoring the similarity by $- \o_B A$ or $A \o_B -$. The \textit{minimum depth} is denoted
by $d(B,A)$; if $B \subseteq A$ has no finite depth, write $d(B,A) = \infty$.  We similarly define minimum odd depth $d_{odd}(B,A)$ and minimum even depth $d_{even}(B,A)$. 

A subring $B \subseteq A$ has \textit{h-depth} $2n-1$ if Eq.~(\ref{eq: dirsum}) is more strongly satisfied
as $A$-$A$-bimodules ($n = 1,2,3,\ldots$).  Note that $B$ has h-depth $2n-1$ in $A$ implies that it has h-depth $2n+1$ (also that it has depth $2n$).  Thus define the minimum h-depth $d_h(B,A)$ (and set this equal to $\infty$ if no such $n \in \N$ exists). Note that h-depth $1$ is the Azumaya-like condition of Hirata in \cite{H}.  The notion of h-depth is studied in \cite{K2012};  by elementary considerations the inequality 
$|d_h(B,A) - d(B,A)| \leq 2$ is satisfied if either the minimum depth or minimum h-depth is finite.
\end{rm}
\end{definition}


\section{Depth of algebras and coalgebras in tensor categories}

In this section, we define depth of algebras and coalgebras in tensor categories.  When applied to algebras and coalgebras in a bimodule tensor category, this definition recovers minimum odd depth defined
in \cite{BKK}  and h-depth defined in \cite{LK2012}.  In particular, a coalgebra in bimodule tensor category is a coring, with depth defined in \cite{HKS}.  An algebra or coalgebra in a finite tensor category is an $H$-module algebra or $H$-module coalgebra with depth defined in \cite{K2014}.  

\subsection{Tensor Category}
\label{subsec-tc}
By a \textit{tensor category} $(\M, \otimes, \mbox{\bf 1})$ we mean an abelian category
$\M$ with \textit{unit object} $\mbox{\bf 1} \in \mbox{\rm Ob}(\M)$ and \textit{tensor product}  $\otimes: \M \times \M \rightarrow \M$, an additive bifunctor (satisfying  distributive laws w.r.t.\ $\oplus$) with \textit{associativity constraint}, a natural isomorphism $$\alpha_{X,Y,Z}: (X \otimes Y) \otimes Z \stackrel{\sim}{\longrightarrow} X \otimes (Y \otimes Z), \ \ \ X,Y,Z \in \M $$
satisfying the pentagon axiom (a commutative pentagon with 4 arbitrary objects in a tensor product grouped together in different ways, see for example \cite[(2.3)]{N}), 
and \textit{unit constraints}, natural isomorphisms $\ell, r$ such that
$$ \ell_X: \mbox{\bf 1} \otimes X \stackrel{\sim}{\longrightarrow} X, r_X: X \otimes \mbox{\bf 1} \stackrel{\sim}{\longrightarrow} X, \ \ \ X \in \M $$
 satisfy the triangle axiom (a commutative triangle with the unit object between two other arbitrary objects in a tensor product  associated in two ways using $\alpha,\ell,r$, \cite[(2.4)]{N}).  The Coherence Theorem of MacLane states that every diagram constructed from associativity and unit constraints commutes.  (Here we are making no requirement of left and right duals satisfying rigidity axioms.)

A \textit{tensor functor} between  tensor categories $(\M, \otimes, \mbox{\bf 1})$ and
$(\M', \otimes', \mbox{\bf 1}')$ is a functor $F: \M \rightarrow \M'$ such that for every $X,Y \in \mbox{\rm Ob}(\M)$, there are isomorphisms
$J_{X,Y}: F(X) \otimes' F(Y) \stackrel{\sim}{\longrightarrow} F(X \otimes Y)$ defining  a natural isomorphism, and $\phi: \mbox{\bf 1}' \stackrel{\sim}{\longrightarrow} F(\mbox{\bf 1})$ is an isomorphism satisfying  a commutative hexagon and two commutative rectangles, see for example
\cite[(2.12),(2.13),(2.14)]{N}. If $F$ is an equivalence of categories, the tensor categories $\M, \M'$ are
said to be \textit{tensor equivalent}.  

\begin{example}
\begin{rm}
Let $R$ be a ring, and ${}_R\M_R$ denote the category of $R$-$R$-bimodules and their bimodule homorphisms (denoted by $\Hom_{R-R}(X,Y)$ or $\Hom ({}_RX_R, {}_RY_R)$). Note that ${}_R\M_R$ has a tensor product $\otimes_R$ and unit object ${}_RR_R$, the natural bimodule structure on $R$ itself.  For example, $\ell_X : R \otimes_R X \stackrel{\sim}{\longrightarrow} X$ is the well-known
natural isomorphism.  This makes $({}_R\M_R, \otimes_R, {}_RR_R)$ into a tensor category.  
\end{rm}
\end{example}
Let $A,R$ are rings, $\M_A$, $\M_R$ their categories of right modules and homomorphisms. Recall that $A$ and $R$ are Morita equivalent rings if $R \cong \End P_A$ for some progenerator $A$-module $P$, if and only if
the categories $\M_R$ and $\M_A$ are equivalent, via the additive functor $-\otimes_R P$. The inverse bimodule of $P$ is denoted without ambiguity by $P^* \cong \Hom (P_A, A_A)$, since
$\Hom (P_A, A_A) \cong \Hom ({}_RP, {}_RR)$ as $A$-$R$-bimodules (by a theorem of Morita \cite{M}).  
\begin{lemma}
\label{lem-tys}
Suppose $T: \M_R \stackrel{\sim}{\longrightarrow} \M_A$ is an equivalence of  categories given
by $T(X) = X \otimes_R P_A$.  Then the categories ${}_R\M_R$ and ${}_A\M_A$ are tensor equivalent via $F({}_RY_R) = P^* \otimes_R Y \otimes_R P$.  
\end{lemma}
\begin{proof} The proof follows from $F(X \otimes_R Y)  = P^* \otimes_R X \otimes_R Y \otimes_R P \cong $ $$ P^* \otimes_R X \otimes_R R \otimes_R Y \otimes_R P\cong P^* \otimes_R X \otimes_R P \otimes_A P^* \otimes_R Y \otimes_R P \cong 
F(X) \otimes_A F(Y).$$   Also $F({}_RR_R) \cong {}_AA_A$.  The functor $F$ is an equivalence with inverse functor $F^{-1}({}_AZ_A) = P \otimes_A Z \otimes_A P^*$.  
\end{proof}

In a tensor category $(\M, \otimes, 1_{\M})$, one says $(B,m,u)$ is an \textit{algebra in $\M$} if the multiplication $m: B \otimes B \to B$,
a morphism in $\M$, satisfies a commutative pentagon \cite[3.9]{N} w.r.t.\ associativity isomorphism  $\alpha_{A,A,A}$ and ``the unit'' $u: 1_{\M} \to A$,
a morphism in $\M$, satisfies two commutative rectangles \cite[3.10]{N} w.r.t.\ the natural isomorphisms $\ell_A, r_A$ in the notation of Subsection~\ref{subsec-tc}. (\textit{Coalgebra} $(B,\cop, \eps)$ is defined dually by coassociative comultiplication $\cop: B \rightarrow B \otimes B$ and counit $\eps: B \to 1_{\M}$ satisfying the counit diagrams.)  
That $B^{\otimes (n)} \| B^{\otimes (n+1)}$ for $n \geq 1$ follows from using the multiplication epi, split by the unit (e.g., see commutative diagram \cite[(3.10)]{N}), or the counit splitting the comultiplication monomorphism. 

\begin{definition}
\label{def-depthtens}
\begin{rm}
Let $B$ be an algebra (or coalgebra) in a tensor category $\M$.  Define $B$ to have depth $1$ if $B \sim 1_{\M}$. 
Define $B$ to have depth $2n+1$ ($n \geq 1$) if $B^{\otimes (n+1)} \| q \cdot B^{\otimes (n)}$
for
some $q \in \N$ ($\Leftrightarrow$ $B^{\otimes (n)} \sim B^{\otimes (n+1)}$) ; in this case, $B$ also has depth $2n+3, 2n+5,\ldots$ by tensoring repeatedly by $- \otimes B$.  If there is a finite $n \in \N$ like this, let $d(B, \M)$ denote the minimum depth (an odd number); otherwise, write $d(B,\M) = \infty$. 
\end{rm}
\end{definition}

\begin{example}
\begin{rm}
Let $A$ be a ring, with tensor category of bimodules ${}_A\M_A$.  An algebra $B$ (or monoid) in ${}_A\M_A$ has unit mapping $u: A \to B$ and multiplication $B \otimes_A B \to B$ satisfying associativity and unital axioms as usual.  This is equivalently a ring extension.  The depth
just defined is the minimum odd depth; i.e., $d(B, {}_A\M_A) = d_{odd}(A,B)$, which is obvious from Definition~\ref{def-depth} (with role reversal). 
\end{rm}
\end{example}
\begin{remark}
\begin{rm}
The reference \cite[3.8]{N} also sketches the definition of modules and bimodules over such algebras, as well as Morita equivalence between two such algebras. For example, a left module over algebra $A$ in tensor category ${}_B\M_B$ is an $A$-$B$-bimodule $N$ as an exercise in applying these ideas.
The category ${}_A\M_B$ is equivalent to the category ${}_A\M$ of left modules over $A$.  If $A'$
is another algebra in ${}_B\M_B$ Morita equivalent in the sense of \cite{N}, then ${}_{A'}\M_B$
is equivalent to ${}_A\M_B$.  This is the case if the ring extensions $B \to A$ and $B \to A'$ are Morita
equivalent in the sense of Section~5, cf. Diagram~(\ref{fig-mixedmodules}).  
\end{rm}
\end{remark}
\begin{example}
\label{ex-coringdepth=hdepth}
\begin{rm}
Let $B = \mathcal{C}$ be an $A$-coring; i.e., a coalgebra (or comonoid) in the tensor category ${}_A\M_A$.  Dual to algebra, there is a comultiplication $\cop: \mathcal{C} \to \mathcal{C} \otimes_A \mathcal{C}$ and counit $\eps: \mathcal{C} \to A$, both $A$-$A$-bimodule homomorphisms, satisfying coassociativity and counit diagrams \cite{BW}. The definition
of minimum depth $d(\mathcal{C}, {}_A\M_A)$ coincides with the depth $d(\mathcal{C},A)$ of corings defined
in \cite[2.1]{HKS}: $d(\mathcal{C}, {}_A\M_A) = d(\mathcal{C},A)$. 

Let $A \supseteq B$ be a ring extension, and $\mathcal{C} = A \otimes_B A$ its Sweedler $A$-coring,
with comultiplication simplifying to $A^{\otimes_B (2)} \to A^{\otimes_B (3)}$, $a_1 \otimes_B a_2 \mapsto a_1 \otimes_B 1 \otimes_B a_2$, and counit $\eps_{\mathcal{C}}: A \otimes_B A \to A$, $a_1 \otimes_B a_2 \mapsto a_1 a_2$ ($a_1,a_2 \in A$).  Comparing with Definition~\ref{def-depth}
and applying cancellations of the type $X \otimes_A A \cong X$,
we see that coring depth of $\mathcal{C}$ recovers h-depth of the ring extension: $d(\mathcal{C}, {}_A\M_A) = d_h(B,A)$.  
\end{rm}
\end{example}
Suppose $k$ is a field, the ground field below for all algebras, coalgebras, modules and  unadorned tensor products in finite tensor categories (including the tensor category of finite-dimensional vector spaces,  Vect$_k$). 
\begin{example}
\begin{rm}
Let $H$ be a finite-dimensional Hopf $k$-algebra; its category of finite-dimensional modules $\M_H$ is
a finite tensor category \cite{EO}.  The tensor $\otimes = \otimes_k$ is defined by the diagonal action,
where 
$V \otimes W$: $(v \otimes w) \cdot h = v h\1 \otimes w h\2$. The unit module is $k_{\eps}$ where
$\eps: H \to k$ is the counit. An algebra $A$ in $\M_H$ is a right $H$-module algebra, which the reader may check satisfies the (measuring) axioms $(ab). h = (a. h\1)(b. h\2)$ and $1_A . h = 1_A \eps(h)$ for all $a,b \in A$ and $h \in H$. A coalgebra $C$ in $\M_H$ is a right $H$-module coalgebra $(C, \cop, \eps_C)$ satisfying 
\begin{equation}
\label{eq: axiom}
\cop(ch) = c\1 h\1 \otimes c\2 h\2, \ \ \eps_C(ch) = \eps_C(c) \eps(h)
\end{equation}
  for all $c \in C, h \in H$. 

The depth $d(A, \M_H)$ and $d(C, \M_H)$ is a linear rescaling of the minimum depth of any object in $\M_H$ defined in
\cite{K2014, HKY, HKS}, not an important difference, though slightly more convenient in formulas given below. 
\end{rm}
\end{example}
\begin{example}
\begin{rm}
Continuing with $H$, the category of right $H$-comodules $\M^H$ is a tensor category, where
$X,Y \in \M^H$ has tensor product $X \otimes Y$ as linear space with comultiplication
$x \otimes y \mapsto x\0 \otimes y\0 \otimes x\1 y\1$.  The unit module is  $k$ with
coaction $1_k \mapsto 1_H$. An algebra $A$ in $\M^H$ has multiplication $m: A \otimes A \to A$
and unit $k \to A$ right $H$-comodule morphisms.  This condition is equivalent to the coaction of $A$, $\rho_A:
A \to A \otimes H$, being an algebra homomorphism (w.r.t.\ the tensor algebra).   Thus $A$
is a right $H$-comodule algebra.  See for example \cite{Ma}.
\end{rm}  
\end{example}
\section{Entwining structures}
In this section we summarise the equalities and inequalities obtained in \cite{HKS} and \cite{HKY} between depths of entwined corings and factorisable algebras on the one hand (in the ``difficult'' tensor bimodule  category) and depth of an $H$-module coalgebra or algebra on the other hand (in a more manageable  finite tensor category \cite{EO}).  We study the quotient module $Q$ of a finite-dimensional Hopf subalgebra pair $R \subseteq H$ in terms of core Hopf ideals, duals and Frobenius extensions, and under conditions of semisimplicity, relative or not.  

Recall that an \textit{entwining
structure} of an  algebra $A$ and  coalgebra $C$ is given by a linear mapping  $\psi: C \otimes A \rightarrow A \otimes C$ (called  the entwining mapping) satisfying two 
commutative pentagons and two triangles (a bow-tie diagram on \cite[p.\ 324]{BW}).  Equivalently,
$(A \otimes C, \id_A \otimes \cop_C, \id_A \otimes \eps_C)$ is an $A$-coring with respect
to the $A$-bimodule structure $a(a' \otimes c){a''} = aa' \psi(c \otimes a'')$
(or conversely defining $\psi(c \otimes a) = (1_A \otimes c) a$) (details in \cite[32.6]{BW} or \cite[Theorem 2.8.1]{CMZ}). 

In more detail, an entwining structure mapping $\psi: C \otimes A \rightarrow A \otimes C$ takes
values usually denoted by $\psi(c \otimes a) = a_{\alpha} \otimes c^{\alpha} = a_{\beta} \otimes c^{\beta}$, suppressing linear sums of rank one tensors, and satisfies the axioms: (for all $a,b \in A, c \in C$)
\begin{enumerate}
\item  $\psi(c \otimes ab) = a_{\alpha}b_{\beta} \otimes {c^{\alpha}}^{\beta}$;
\item $\psi(c \otimes 1_A) = 1_A \otimes c$;
\item $a_{\alpha} \otimes \cop_C(c^{\alpha}) = {a_{\alpha}}_{\beta} \otimes {c\1}^{\beta} \otimes {c\2}^{\alpha}$
\item $a_{\alpha} \eps_C(c^{\alpha}) = a \eps_C(c)$,
\end{enumerate}
which is equivalent to two commutative pentagons (for axioms 1 and 3) and two commutative
triangles (for axioms 2 and 4), in an exercise.

\subsection{Doi-Koppinen entwinings \cite{BW, CMZ}} Let $H$ be a finite-dimensional Hopf algebra.  Suppose $A$ is an algebra in the tensor category of right $H$-comodules, equivalently, $A$ is a right $H$-comodule algebra.  Moreover, let
$(C, \cop_C, \eps_C)$ be a  coalgebra in the tensor category
$\mathcal{M}_H$,  right $H$-module coalgebra as noted in the example above in Section~2.  Of course, if $H = k$ is the trivial one-dimensional Hopf algebra, $A$ may be any $k$-algebra and $C$
any $k$-coalgebra.  

\begin{example}
\begin{rm}
The Hopf algebra $H$ is right $H$-comodule algebra over itself, where $\rho = \cop$.  Given a Hopf subalgebra
$R \subseteq H$ the quotient module $Q$ defined as $Q = H/ R^+H$. Note that $Q$ is a right $H$-module coalgebra. So is  $(H, \cop, \eps)$
trivially a right $H$-module coalgebra.  The canonical epimorphism $H \rightarrow Q$ denoted by $h \mapsto \overline{h}$ is an epi of right $H$-module coalgebras. The module $Q_H$ is cyclic with generator $\overline{1_H}$.
\end{rm}
\end{example}   

The mapping $\psi: C \otimes A \rightarrow A \otimes C$ defined by 
$\psi(c \otimes a) = a\0 \otimes c a\1$ is an entwining  (the \textit{Doi-Koppinen entwining}
\cite[33.4]{BW}, \cite[2.1]{CMZ}). 
From the equivalence of corings with entwinings, it follows that $A \otimes C$ has $A$-coring structure
\begin{equation}
\label{eq: comodulealgdiagonal}
a(a' \otimes c)a'' = aa' {a''}\0 \otimes c {a''}\1
\end{equation}
which defines the bimodule ${}_A(A \otimes C)_A$.  The coproduct is given by $\id_A \otimes \cop_C$
and the counit by $\id_A \otimes \eps_C$.   

Note that  Eq.~(\ref{eq: comodulealgdiagonal}) above, and Eq.~(\ref{eq: modcat}) below, exhibit the category $\mathcal{M}_A$ as a module category over $\mathcal{M}_H$ \cite{EO}.
 
\begin{prop}
\label{prop-comodalgentwinedepth} \cite[Prop.\ 4.2]{HKS}
The depth of the $A$-coring $A \otimes C$ (of a Doi-Koppinen entwining) and the depth of the $H$-module coalgebra $C$ are related by
$d(A \otimes C, {}_A\M_A) \leq  d(C, \mathcal{M}_H)$.  
\end{prop}
\begin{proof}
One notes that $(A \otimes C)^{\otimes_A (n)} \cong A \otimes C^{\otimes (n)}$ as $A$-$A$-bimodules  via cancellations of
the type $X \otimes_A A \cong X$. Keeping track of the right $A$-module structure on $A \otimes C^{\otimes (n)}$, one shows that it
is given by
\begin{equation}
\label{eq: modcat}
(a \otimes c_1 \otimes \cdots \otimes c_n)b = ab\0 \otimes c_1 b\1 \otimes \cdots \otimes c_n b_{(n)}.
\end{equation}
 If  $d(C, \mathcal{M}_H) = n$, then $C^{\otimes (n)} \sim C^{\otimes (n+1)}$ in the finite tensor category $\mathcal{M}_H$.  Applying an additive functor,
it follows that $A \otimes C^{\otimes (n)} \sim A \otimes C^{\otimes (n+1)}$ as $A$-bimodules.
Then applying the isomorphism just above and Definition~\ref{def-depthtens} obtains the inequality in the proposition. 
\end{proof}

For example, if $A = H$,  and $C$ a right $H$-module coalgebra,
the Doi-Koppinen entwining mapping $\psi: C \otimes H \rightarrow  H \otimes C$ is of course $\psi(c \otimes h) = h\1 \otimes c h\2$. The associated $H$-coring $H \otimes C$ has coproduct $\id_H \otimes \cop_C$ and counit
$\id_H \otimes \eps_C$ with $H$-bimodule structure: ($x,y,h \in H, c \in C$) 
\begin{equation}
\label{eq: coring}
x(h \otimes c) y = xh y\1 \otimes c y\2
\end{equation}

\begin{cor}
\label{prop-entwinedepth}\cite[Prop.\ 3.2]{HKS}
The depth of the $H$-coring $H \otimes C$ and the depth of the $H$-module coalgebra $C$ are related by
$d(H \otimes C, H) = d(C, \mathcal{M}_H) $.  
\end{cor}
\begin{proof}
This follows immediately from the proposition, but the proof reverses as follows. 
If $d(H \otimes C, {}_H\M_H) = 2n+1$, so that $H \otimes C^{\otimes (n)} \sim H \otimes C^{\otimes (n+1)}$ as $H$-$H$-bimodules, apply the additive functor $k \otimes_H -$ to the similarity
and obtain the similarity of right $H$-modules, $C^{\otimes (n)} \sim C^{\otimes (n+1)}$.  Thus $d(C, \mathcal{M}_H)  \leq d(H \otimes C, {}_H\M_H)$ as well.   
\end{proof}

The corollary applies as follows. Let $K \subseteq H$ be a left coideal subalgebra of a finite-dimensional Hopf algebra;
i.e., $\cop(K)\subseteq  H \otimes K$.  Let $K^+$ denote the kernel of the counit restricted to $K$.
Then $K^+H$ is a right $H$-submodule of $H$ and a coideal by a short computation given
in \cite[34.2]{BW}.  Thus $Q := H/K^+H$ is a right $H$-module coalgebra (with a right $H$-module
coalgebra epimorphism $H \rightarrow Q$ given by $h \mapsto h + K^+H := \overline{h}$).
The $H$-coring $H \otimes Q$ has grouplike element $1_H \otimes \overline{1_H}$; in fact,
\cite[34.2]{BW} together with \cite{S} shows that this coring is Galois:
\begin{equation}
\label{eq: map}
H \otimes_K H \stackrel{\cong}{\longrightarrow} H \otimes Q
\end{equation}
via $x \otimes_R y \mapsto x y\1 \otimes \overline{y\2}$, an $H$-$H$-bimodule isomorphism.  That $H_K$ is faithfully flat follows from Skryabin's Theorem  \cite{S} that $K$ is a Frobenius algebra and  $H_K$ is free.  Note that an inverse to (\ref{eq: map}) is given
by $x \otimes \overline{z} \mapsto xS(z\1) \otimes_K z\2$ for all $x,z \in H$. 

From Proposition~\ref{prop-entwinedepth}, Eq.~(\ref{eq: map}) and Example~\ref{ex-coringdepth=hdepth} we note the first statement below. The second statement is proven similarly as shown in \cite{K2014}.     

\begin{cor}
\label{cor-cue}\cite[Corollary 3.3]{HKS}\cite[Theorem 5.1]{K2014}
The h-depth of $K \subseteq H$ 
is related to the depth of $Q$ in $\M_H$ by 
\begin{equation}
\label{eq: cue}
d_h(K,H) = d(Q, \mathcal{M}_H).
\end{equation}  
If $R$ is a Hopf subalgebra of $H$, the following holds:
\begin{equation}
\label{eq: cue-even}
d_{even}(R,H) = d(Q,\mathcal{M}_R) + 1
\end{equation} 
\end{cor}

The following is of use to computing depth graphically from a bicolored graph in case $R$ and $H$ are semisimple $\C$-algebras. Let
$U$ denote the functor of restriction-induction, i.e., $U = \Ind^H_R \Res^H_R : \M_H \rightarrow \M_H$.
\begin{prop}
The depth $d(Q, \M_H) = 2n+1$ is the least $n$  for which $U^n(k) \sim U^{n+1}(k)$.  
\end{prop}
\begin{proof}
Recall  that $Q \cong k \otimes_R H$ and for any module $M_H$,
$U(M) \cong M \otimes Q$ (tensor in $\M_H$) \cite{K2014}.  It follows by induction that $Q^{\otimes (n)} \cong U^n(k)$.  
\end{proof}
Note that decomposing $Q$ into its projective-free direct summand $Q_0$ and projective summand
$Q_1$, such that $Q = Q_0 \oplus Q_1$, leads to the following from the fact that projectives
form an ideal in the Green ring of $H$.
\begin{prop}
The depth of the Hopf subalgebra, $d_h(R,H) < \infty$ if and only if the module depth $d(Q_0, \M_H) < \infty$.
\end{prop}
\begin{proof}
For the statement and proof of this proposition,  we apply the extended definition of module depth of any finitely generated module $X \in \M_H$
in terms of the depth $n$ condition, $T_n(X) \sim T_{n+1}(X)$ where $T_n(X) = X \oplus \cdots \oplus X^{\otimes (n)}$
\cite{K2014}. Since $T_n(X) \| T_{n+1}(X)$, any projective module $Y$ has finite depth, as there are a finite number of isoclasses of projective indecomposables.  But $Y \otimes M$ is projective as well
for any $M \in \M_H$.  Then  $Q^{\otimes (n)} = Q_0^{\otimes (n)} \oplus Q_1^{\otimes (n)} \oplus$ mixed terms of $Q_0, Q_1$, which are all projective.   Thus $d_h(R,H) < \infty$ $\Leftrightarrow$ $Q^{\otimes (n)} \sim Q^{\otimes (n+1)}$ as $H$-modules for some $n \in \N$, which  implies  that the summand $Q_0$ has finite depth by \cite[Lemma 4.4]{K2014}. Conversely, if $T_n(Q_0) \sim
T_{n+1}(Q_0)$ as $H$-modules, from $T_i(Q) \| T_{i+1}(Q)$,  we obtain that
$T_{n+m}(Q) \sim T_{n+m+1}(Q)$,  equivalently $Q^{\otimes (n+m)} \sim
Q^{\otimes (n+m+1)}$, where $m$ is the number of distinct isoclasses of projective indecomposables.  
\end{proof}

\subsection{Semisimple and separable extensions}  Recall that any ring extension $A \supseteq B$ is said to be a \textit{right semisimple extension} if any right $A$-module $N$ is relative projective, i.e., $N \| N \otimes_B A$ as $A$-modules. More strongly, a ring extension $A \supseteq B$ is said to be a \textit{separable extension} if for any right $A$-module $M$, the multiplication epimorphism
$\mu_M: M \otimes_B A \rightarrow M$ splits \cite{HS}, which also generalizes the straightforward  notion of left semisimple extension.  The following theorem is a \textit{relative Maschke theorem} characterizing   semisimple extensions of finite-dimensional Hopf algebras $R \subseteq H$.  We freely use the notation $Q = H/R^+H$ and ground field $k$ developed above.
\begin{theorem}
The Hopf subalgebra pair $R \subseteq H$ is a right (or left) semisimple extension $\Leftrightarrow$ $k_H \| Q_H$ $\Leftrightarrow$ $k_H$ is $R$-relative projective $\Leftrightarrow$ there is  $q \in Q$ such that $\eps_Q(q) \not = 0$ and $qh = q \eps(h)$ for every $h \in H$ $\Leftrightarrow$  $\exists \, s \in H: \ sH^+ \subseteq R^+H$ and $\eps(s) = 1$ $\Leftrightarrow$ $H$ is a separable extension of $R$. 
\end{theorem}
\begin{proof}
The counit of $Q$, given by $\eps_Q(\overline{h}) = \eps(h)$ for $h \in H$, is always $R$-split by
$1 \mapsto \overline{1_H}$. If all modules are relative projective, it follows that $\eps_Q$ $H$-splits, so $k_H$ is isomorphic to a direct summand of $Q_H$.  Conversely, if $Q_H \cong k_H \oplus {Q'}_H$, then any $H$-module
$N$ satisfies by \cite[Lemma 3.1]{K2014} $$N \otimes_R H \cong N_. \otimes Q_. \cong N \oplus (N_. \otimes {Q'}_. )$$
since $N_. \otimes k_. \cong N_H$.  Thus, $N$ and all $H$-modules are relative projective.  

If $\eps_Q: Q \rightarrow k$ is split by an $H$-module mapping $k_H \rightarrow Q_H$, where $1 \mapsto q$
under this mapping, then $q$ satisfies the integral-like condition of the theorem as well as $\eps_Q(q) = 1$.  Moreover,  $q = \overline{s} \not = \overline{0}$, satisfies $\eps(s) = 1$ and $sh - s\eps(h) \in R^+H$ for all $h \in H$, but all elements of $H^+$ are of the form $h - \eps(h)1_H$.  

If an element $s \in H$ exists satisfying the conditions of the theorem, for any $H$-module $M$,
the epi $\mu_M: M \otimes_R H \rightarrow M$ is split by $m \mapsto m S(s\1) \otimes_R s_2$.  This is also seen from a commutative triangle using $M \otimes_R H \stackrel{\cong}{\longrightarrow} M_. \otimes Q_.$ and the mappings in \cite[Lemma 3.1]{K2014}. Note that $S(s\1) \otimes_R s_2$ is a separability element, for given any $ h \in H $, $sh = \eps(h)s -\sum_i x_i h_i$ for some $x_i \in R^+,
h_i \in H$.  Applying $\pi (S\otimes \id)\cop$ (where $\pi: H \otimes H \rightarrow H\otimes_R H$
is the canonical epimorphism) to this equation: $S(h\1) S(s\1) \otimes_R s\2 h\2 = $
$$ \eps(h) S( s\1) \otimes_R s\2 - \sum_i S({h_i}\1) S({x_i}\1) \otimes_R {x_i}\2 {h_i}\2 $$
$$ = \eps(h) S(s\1) \otimes_R s\2.$$
Then $hS(s\1) \otimes_R s_2 = S(s\1) \otimes_R s\2 h$ for all $h \in H$ follows from a standard
application of $h\1S(h\2) \otimes h\3 = 1 \otimes h$.  
\end{proof}
Note that if $R = k1_H$, the theorem recovers the extended Maschke's theorem for Hopf algebras (e.g.,
\cite[Ch.\ 2]{M}), since
$R^+ = \{ 0 \}$, $Q = H$ and $q$ or $s$ are integral elements of $H$ with nonzero counit.  
For example, if $Q^{\otimes (n)}$
is projective as an $H$- or $R$-module for any $n \in N$, it follows from this theorem that $R$ is semisimple, since $k_R \| Q \| \cdots \| Q^{\otimes (n)}$.  

Let $t_R, t_H$ denote nonzero right integrals in $R,H$, respectively, for the proof of the  corollary below.  
\begin{cor}
Suppose $H \supseteq R$ is a semisimple extension of finite-dimensional Hopf algebras.  Then
\begin{enumerate}
\item the modular functions of $H$ and $R$ satisfy $m_H |_R = m_R$;
\item the Nakayama automorphisms of $H$ and $R$ satisfy $\eta_H |_R = \eta_R$;
\item the extension $H \supseteq R$ is an ordinary Frobenius extension.
\end{enumerate}
\end{cor}
\begin{proof}
Suppose $s \in H$ satisfies the conditions of the theorem, $\eps(s) = 1$ and $sH^+ \subseteq R^+H$.
By \cite[Lemma 3.2]{K2014}. the quotient module $$Q \stackrel{\sim}{\longrightarrow} t_R H,$$ which
sends $q = \overline{s} \mapsto t_R s$.  Then $t_RsH^+ \in t_R R^+ H = \{ 0 \}$, i.e., $t_Rs$ is a nonzero integral in $H$.  Without loss of generality, set $t_H = t_R s$.  Then for all $r \in R$,
$$ m_H(r)t_H = rt_H = rt_R s = m_R(r) t_H,$$
from which it follows that $m_H$ restricts on $R$ to the modular function of $R$, $m_R$.  

Recall that finite-dimensional Hopf subalgebra pairs such as $H \supseteq R$ are $\beta$-Frobenius extensions (Fischman-Montgomery-Schneider) with $$\beta(r) = r \leftharpoonup m_H * m_R^{-1} = 
\eta_R(\eta^{-1}_H(r)).$$ See \cite{K} or \cite{SY} for  textbook coverages of the full details. 
Consequently, $\eta_H(r)= \eta_R(r)$, $m_H(r) = m_R(r)$ and $\beta(r) = r$ for all $r \in R$. 
\end{proof}
The hypothesis of semisimplicity that removes the twist in the Frobenius extension of Hopf algebra substantially uncomplicates the associated induction theory. 

\subsection{Depth of Hopf subalgebras from right or left quotient modules}  Let $R \subseteq H$ be a Hopf subalgebra pair where $H$ is finite-dimensional, and $R^+ = \ker \eps \cap R$.  The right quotient $H$-module $Q := H/R^+H$ controls induction of right $H$-module restricted to $R$-modules as follows: $\forall \ M \in \M_H,$
\begin{equation}
\label{eq: inductQ}
M \otimes_R H \stackrel{\cong}{\longrightarrow} M_. \otimes Q_., \ \ m \otimes_R h \mapsto mh\1 \otimes h\2
\end{equation}
with inverse mapping given by $m \otimes \overline{h} \mapsto mS(h\1) \otimes_R h\2$ where $S: H \rightarrow H$ denotes the antipode of $H$.  At the same time, the $k$-dual of the left quotient $H$-module $\mathcal{Q} := H/HR+$ controls the  coinduction of right $H$-modules restricted to $R$-modules in
a somewhat similar way: $\forall \ M \in \M_H,$
\begin{equation}
\label{eq: coinductQ}
M_. \otimes {\mathcal{Q}}^*_. \stackrel{\cong}{\longrightarrow} \Hom (H_R, M_R), \ \ m \otimes q^* \mapsto (h \mapsto mh\1 q^*( \overline{h\2}))
\end{equation}
Both Eqs.~(\ref{eq: inductQ}) and~(\ref{eq: coinductQ}) are first recorded in \cite[Ulbrich]{U}; we use the notation for cosets $\overline{h}$ for both coset spaces 
$Q$ and $\mathcal{Q}$. 

The following is then a consequence of Eqs.~(\ref{eq: inductQ}) and~(\ref{eq: coinductQ}).
As mentioned above, $H \supseteq R$ is always a twisted (``beta'') Frobenius extension, with a twist automorphism $\beta: R \rightarrow R$ given by a relative
modular function or a relative Nakayama automorphism.  
If the twist is trivially the identity on $R$, the Hopf subalgebra is an ordinary Frobenius extension: see subsection~\ref{subsect-tower} of this paper for the definition. This hypothesis on $H \supseteq R$ allows us to prove the following.  
\begin{prop}
If $H \supseteq R$ is a Frobenius extension, then $\mathcal{Q}^* \cong Q$ as right $H$-modules.     
\end{prop} 
\begin{proof}
This follows from the characterization of Frobenius extension:  for each right $R$-module $N$,
\begin{equation}
\label{eq: frob}
N \otimes_R H \cong \Hom (H_R, N_R).
\end{equation}
Now apply this and the display equations above to $N = M = k_{\eps}$.
\end{proof}
Recall that $H$ and $R$ are Frobenius algebras: let $A$ be any Frobenius algebra. Then there are one-to-one correspondences of right
ideals with left ideals of  $A$ via the correspondence $I \mapsto \ell(I) :=  \{ a \in A: \ aI = 0 \}$ for
every right ideal $I $ of $A$, and inverse correspondence $J \mapsto r(J) := \{ a \in A : \ Ja = 0 \}$ for every left ideal $J$ of $A$.  The following comes from the basic fact that $\ell(I) \cong \Hom ((A/I)_A, A_A)$
and $r(J) \cong \Hom ({}_A(A/J), {}_AA)$.  See \cite[Lam II]{Lam}.  
\begin{prop}
Let $t_R$ denote a nonzero right integral in $R$, a Hopf subalgebra of $H$ as above.  Then
$\ell(R^+H) = Ht_R$,  $$\Hom ({}_H(H/Ht_R), {}_HH) \cong R^+H$$
and $\Hom (Q_H, H_H) \cong Ht_R$.  If $H$ is a symmetric algebra, the $k$-duals $Q^* \cong Ht_R$
and $\mathcal{Q}^* \cong t_R H$. 
\end{prop}
\begin{proof}
Note that $Ht_R R^+H = 0$.  From  \cite[3.2]{K2014} $Q \cong t_R H$ and $\dim Q = \dim H / \dim R$. By definition of $Q$, $\dim Q = \dim H - \dim R^+H$;  similarly
$$\dim Ht_R = \dim \mathcal{Q} = \dim H / \dim R.$$  For a Frobenius algebra $A$,
we know that $\dim \ell(I) = \dim A - \dim I$ \cite{Lam}.  Setting $A = H$, it follows from dimensionality that
$Ht_R = \ell(R^+H)$.  The next two isomorphisms are  applications of $r(\ell(I) = I$ and $\ell(r(J) = J$.
 The last statement follows from $$\Hom (M_A, A_A) \cong M^*$$
 as left $A$-modules, for every $A$-module $M$, for a symmetric algebra $A$ (and a similar statement for left $A$-modules, see \cite{Lam}).   
\end{proof} 

The equivalent problems in Section~1 have a third equivalent formulation based on elementary considerations using Eq.~(\ref{eq: oplus}):

\begin{problem}
Is there an $n \in \N$ such that the composition $$\Hom (Q^{\otimes (n)}, Q^{\otimes (n+1)}) \otimes_{\End Q^{\otimes (n)}} \Hom ( Q^{\otimes (n+1)}, Q^{\otimes (n)}) \longrightarrow \End Q^{\otimes (n+1)} $$
is surjective?
\end{problem}
Either $R$-modules or $H$-modules suffice above.  If we assume that $H \supseteq R$ is an ordinary Frobenius extension however, the following interesting isomorphisms of Hom-groups over $H$ exist.
Note that for any $H$-module $M$, there is a subring pair $\End M_H \subseteq \End M_R$.  
\begin{prop}
There are $\End Q^{\otimes (n)}_H := E$-module isomorphisms,
$$ \Hom (Q^{\otimes (n)}_H, Q^{\otimes (n+1)}_H) \cong \End Q^{\otimes (n)}_R \cong
\Hom  ( Q^{\otimes (n+1)}_H, Q^{\otimes (n)}_H)$$
(right and left $E$-modules respectively). 
\end{prop}
\begin{proof}
The second isomorphism follows from Eq.~(\ref{eq: inductQ}) and the hom-tensor adjoint isomorphism \cite[20.6]{AF}.  The first isomorphism requires additionally the  fact for any Frobenius extension $H \supseteq R$ with modules
$M_H$ and $N_R$: 
\begin{equation}
\label{eq: factoid}
\Hom (M_H, N \otimes_R H_H) \cong \Hom (M_R, N_R)
\end{equation}
which follows from a natural isomorphism $\Hom (H_R, N_R) \cong N \otimes_R H$ as right $H$-modules, and the hom-tensor adjoint isomorphism.  
\end{proof}
It is worth remarking that the tensor powers of $Q$ are also $H$-module coalgebra quotients, since they are pullbacks via $\cop^n: H \rightarrow H^{\otimes (n)}$ of the quotient module of the Hopf subalgebra
pair $R^{\otimes (n)} \subseteq H^{\otimes (n)}$, which is isomorphic to $Q^{\otimes (n)}$
as $H^{\otimes (n)}$-modules. 
\subsection{Core Hopf ideals of a Hopf subalgebra pair}
Let $R \subseteq H$ be a finite-dimensional Hopf subalgebra pair.  
We continue the study begun in \cite{HKY} relating the depth of a quotient module $Q$ to its descending chain of annihilator ideals of its tensor powers:
\begin{equation}
\Ann Q \supseteq \Ann (Q \otimes Q) \supseteq \cdots \supseteq \Ann Q^{\otimes (n)} \supseteq \cdots.
\end{equation}
The chain of ideals are either contained in $R^+$ or $H^+$ depending on whether $Q$ is considered an $R$-module or $H$-module (as in Corollary~\ref{cor-cue}). By classical theory recapitulated in \cite[Section~4]{HKY}, for some $n \in \N$
we have $\Ann Q^{\otimes (n)} = \Ann Q^{\otimes (n+m)}$ for all integers
$m \geq 1$:  this ideal $I$ is a Hopf ideal, indeed the maximal Hopf ideal contained
in $\Ann Q$.   Let $\ell_Q$ denote the least $n$ for which this stabilization of the descending chain of annihilator ideals takes place; call $\ell_Q$ the \textit{length} of the annihilator chain of tensor powers of the quotient module.  This may be nuanced by
$\ell_{Q_R}$ or $\ell_{Q_H}$ depending on which module $Q$ is being considered: since for any module $M_H$ we have 
$\Ann M_R = \Ann M_H \cap R$, it follows that
\begin{equation} 
\ell_{Q_R} \leq \ell_{Q_H}.
\end{equation}

Let $S_1, \ldots, S_t$ be the simple composition factors of $Q$ or one of its tensor powers; by elementary considerations with the composition series of $Q^{\otimes i}$, we note that
\begin{equation}
I \subseteq \cap_{j=1}^t \Ann S_j,
\end{equation}
 in particular, if some $Q^{\otimes i}$ contains all simples (of $R$ or $H$), $I \subseteq J_{\omega}$,
the (Chen-Hiss \cite{CH}) Hopf radical ideal, since $J_{\omega}$ is the maximal Hopf ideal in the radical which is the intersection of the annihilator ideals of all simples.  If one simple is projective, the corresponding $J_{\omega} = 0$ by a result in \cite{CH}, whence $Q$ is conditionally faithful, i.e.,
$Q^{\otimes (n)}$ is faithful for some $n \in \N$  \cite{HKY}.   

Recall that the \textit{core} of a subgroup $U \leq G$ is $N := \cap_{g \in G}\,  gUg^{-1}$, and is the maximal
normal subgroup of $G$ contained in $U$.  
\begin{prop}
Suppose $H$ is a group algebra $kG$ and $R$ is a group algebra $kU$, where $U \leq G$ is a subgroup pair.  Then $I$ is determined by the core $N$ as follows: $I_H = kN^+ H$ and $I_R = kN^+ R$.
\end{prop}
\begin{proof}
Note that $kN^+H = H kN^+$ is a Hopf ideal since $N$ is normal in $G$.  An arbitrary element in $Q$
is the coset $Ug$ annihilated by $1 - n$ for any $n \in N$, since $N \subseteq U$.  Then $KN^+H \subseteq I$, since $I$ is maximal Hopf ideal in the annihilator of $Q$.  Conversely, the Hopf ideal 
$I = k\tilde{N}^+H$ for some normal subgroup  $\tilde{N} \triangleleft G$ by a result in \cite{PQ}. 
Since $1 - \tilde{n}$ annihilates each $Ug$, it follows that $\tilde{N} \subseteq U$, whence
$\tilde{N} = N$ by maximality.   
\end{proof}
Due to the proposition, we propose calling the pair of Hopf ideals
$I = \Ann Q^{\otimes \ell_{Q_H}}$ and $I \cap R = \Ann Q^{\otimes \ell_{Q_R}}$ the \textit{core Hopf ideals} of the Hopf subalgebra $R \subseteq H$.

Note that \cite[Prop.\ 4.3]{HKY} is equivalent to the inequality 
\begin{equation}
\label{ineq1}
2 \ell_{Q_R} +1 < d_{even}(R,H),
\end{equation}
true without further conditions on $H$ and $R$, since the even depth of $Q$, determined from similarity of tensor powers of $Q$ as $R$-modules, results in
equal annihilator ideals: see the first statement in  Proposition~\ref{prop-gem}. 
Similarly, considering the $H$-module $Q$ and h-depth instead, we note that
\begin{equation}
\label{ineq2}
2 \ell_{Q_H} + 1 \leq d_h(R,H)
\end{equation}
Now we make use of the second statement in Proposition~\ref{prop-gem}:

\begin{theorem}
\label{th-precise}
Suppose $R$ is a semisimple Hopf algebra, then $$d_{even}(R,H) = 2 \ell_{Q_R} +2.$$ If moreover $H$ is semisimple, then $d_h(R,H) = 2 \ell_{Q_H} +1$.
 \end{theorem}
\begin{proof}
Semisimple rings satisfy the equal-annihilator-similar-module condition in Proposition~\ref{prop-gem}. The definition~\ref{def-depthtens} of depth
of $Q$ depends on similarity of tensor powers of $Q$ and involves a rescaling of 1 plus a factor of $2$ with respect  to $\ell_Q$.  The rest follows from the inequalities~(\ref{ineq1}) and~(\ref{ineq2}); see also \cite[Theorem 5.1]{K2014} for 
$d_{even}(R,H) = d(Q, \M_R) + 1$.    
\end{proof}

For a semisimple Hopf subalgebra pair, also note the equalities that follow from Def.~\ref{def-depthtens} and Prop.~\ref{prop-gem}: 

\begin{eqnarray}
d(Q, \M_H) & =  & 2 \ell_{Q_H} + 1 \\
d(Q, \M_R) &  = & 2 \ell_{Q_R} + 1.
\end{eqnarray}

For semisimple Hopf algebra-subalgebra pairs, these formulas put the length $\ell_Q$ of the annihilator chain of tensor powers of $Q$ in close relation to diameter of same colored points in the bicolored graph \cite{BKK} as well as the base size or minimal number of ``conjugates'' of the Hopf subalgebra intersecting in the core,
cf. \cite{FKR, BKK}.  

    A general finite-dimensional Hopf subalgebra pair $R \subseteq H$ may sometimes reduce to the hypothesis of the  previous theorem via the following proposition, which extends \cite[Corollary 4.13]{HKS} from the core of a subgroup-group algebra pair.  

\begin{prop}
Suppose $I$ denotes the maximal Hopf ideal in the annihilator ideal of $Q = H/R^+H$; let $J = R \cap I$ denote the restricted Hopf ideal in $R$.  Then h-depth $d_h(R,H) = d_h(R/J, H/I)$.  Similarly,
minimum even depth satisfies $d_{even}(R,H) = d_{even}(R/J, H/I)$.  
\end{prop}
\begin{proof}
Note that $d_h(R,H) = d(Q, \M_H) $ by Corollary~\ref{cor-cue}, and $d(Q, \M_H) = d(Q, \M_{H/I})$
by \cite[Lemma 1.5]{HKS}. Note that $R/J \into H/I$ is a Hopf subalgebra pair with quotient module
isomorphic to $Q$ by a Noether isomorphism theorem.  Then $d_h(R/J, H/I) = d(Q, \M_{H/I})$.   
\end{proof}

\subsection{Quotient module for the permutation group series}
It is interesting at this point to compute the quotient module $Q$ for the inclusion $\C S_n \subseteq \C S_{n+1}$
of permutation group algebras.  Notice that the proposition below implies that the character $\chi_Q = \chi_1 + \chi_t$, where $\chi_1$ is the principal character and $\chi_t$ is the  character of the standard irreducible representation $(n,1)$.  
\begin{prop}
The quotient module $Q = \C [S_n / S_{n+1}]$ is isomorphic to the standard representation of $S_{n+1}$ on $\C^{n+1}$. 
\end{prop}
\begin{proof}
Recall the Artin presentation of $S_{n+1}$ with generators $\sigma_i = (i \  \ i+1)$ for $i = 1,\ldots,n$ and relations $$\sigma_i \sigma_{i+1} \sigma_i = \sigma_{i+1} \sigma_i \sigma_{i+1},\  \sigma_i \sigma_j = \sigma_j \sigma_i,\  \sigma_i^2 = 1$$
for all $| i - j | \geq 2$.  Note that $\sigma_1,\ldots,\sigma_{n-1} \in S_n$.  
An ordered basis for $Q$ is given by $$\bra S_n \sigma_n \sigma_{n-1} \cdots \sigma_1,\  S_n\sigma_n\cdots \sigma_2,\ \ldots,\ S_n \sigma_n,\  S_n \ket$$
This ordered basis maps onto the ordered basis $\bra e_1,\ldots,e_{n+1}\ket$ of the $S_{n+1}$ representation space $\C^{n+1}$ via the canonical order-preserving mapping. This mapping is an $S_{n+1}$-module
isomorphism, since $\sigma_i$ exchanges $e_i$ and $e_{i+1}$ as it does $S_n\sigma_n \cdots \sigma_i$ and $S_n\sigma_n\cdots\sigma_{i+1}$, respectively, (here we use $\sigma_i^2 = 1$), and it leaves the other basis elements fixed, since $\sigma_i$ commutes with $\sigma_{i+2}$ and/or $\sigma_{i-2}$ (here we also use $\sigma_i \sigma_{i-1} \sigma_i = \sigma_{i-1} \sigma_i \sigma_{i-1}$)
etc. while $\sigma_i \in S_n$ for $i < n$.  In more detail, note that
$$(S_n\sigma_n \cdots \sigma_i \sigma_{i-1}) \sigma_i = S_n \sigma_n \cdots \sigma_{i+1} \sigma_{i-1} \sigma_i \sigma_{i-1} = S_n\sigma_n \cdots \sigma_{i-1} $$
The rest of the proof is routine.  (A second proof follows from $Q \cong U(\mbox{\bf 1}) \cong \Ind^{S_{n+1}}_{S_n}\mbox{\bf 1}$ and Young diagram branching rule of adding a box.)   
\end{proof}

Since $S_n \subseteq S_{n+1}$ is corefree, i.e., the core of the subgroup is trivial, it follows that the \textit{character $\chi_Q$ is faithful} (equivalently,
 the annihilator idea of $Q$ does not contain a nonzero Hopf ideal $\Leftrightarrow$ the representation of $\C G$ restricted to $G$ has trivial kernel) \cite[4.2]{K2014}. 
The Burnside-Brauer Theorem \cite[p.\ 49]{Is} implies for the character $\chi_Q$ that each irreducible character of $S_{n+1}$ is a constituent of its powers up to $\chi_Q^n$, since $\dim Q = n+1$. This implies that $d(Q, \M_{S_{n+1}}) \leq n$ by reasoning along the lines of Example~\ref{ex-artinian}.  Indeed $d(Q, \M_{S_{n+1}}) = n$ follows from Corollary~\ref{cor-cue} and
the graphical computation $d_h(S_n, S_{n+1}) = 2n+1$ in \cite{K2014}.  

We mention the theorem in \cite{IM}, that hooks generate the Green ring of a permutation group, as the full picture to  the discussion above.  
\begin{theorem}\cite[Marin]{IM}
The representation ring $A(\C S_{n+1})$ is generated by the representations $\Lambda^k \C^{n+1}$
for $0 \leq k \leq n$.  
\end{theorem}
\begin{remark}
\begin{rm}
Recall the notion of order of a module $V_H$ over a semisimple Hopf algebra $H$.  The order $ord(V)$ is the least natural number $n$ such that $V^{\otimes (n)}$ has nonzero invariant subspace, i.e., 
$\dim (V^{\otimes (n)})^H \neq 0$. For example, $ord(Q_{S_{n+1}}) = 1$ since $\chi_Q = \chi_1 + \chi_t$.  For general semisimple Hopf subalgebra pairs $H \supseteq R$ with quotient $Q$, one might conjecture that $ord(Q) \leq \ell_Q$, 
since  order of $Q$ and  $\ell_Q$ are both bounded above by the degree $d$ of the minimal polynomial of $\chi_Q$ in the character ring of $H$ (or $H/J$-modules where $J = \Ann Q^{\otimes \ell_Q}$  see \cite[chs.\ 4,5, p.\ 37]{KSZ} and 
\cite[2.3]{BKK}, respectively).  However, \cite[p.\ 32]{KSZ} computes the order of a certain induced module $V$  over the semidirect product group algebra $H = \C[\Z_{11}] \# \C[\Z_5]$ to be $ord(V) = 3$: with $R= \C[\Z_q]$, in fact $V \cong Q_H$. We deduce that $d(Q, \M_H) = 3$, since $d_h(R,H) = 1$ forces
$R = H$ by \cite[Cor.\ 3.3]{K2014}),  and $\ell_{Q_H} = 1$, since $R$ is a normal Hopf subalgebra in $H$: so in general $ord(Q) \not \leq \ell_Q$.  
\end{rm}
\end{remark}

   
\section{Factorisable algebras}  An \textit{algebra factorisation} of a unital (associative) algebra $C$ into two unital subalgebras $A$ and $B$ occurs when the multiplication mapping $B \otimes A \stackrel{\sim}{\longrightarrow} C$ is a $B$-$A$-bimodule isomorphism \cite{BW, CMZ}.  Conversely, the algebra $C$ may be constructed from $B$ and $A$ as a \textit{twisted tensor product} (denoted
by $B \otimes^R A$) as follows: linearly
$C = B \otimes A$ with multiplication given by the structure mapping $R: A \otimes B \to B \otimes A$,
values denoted by $R(a \otimes b) = b^r \otimes a^r$ or $b^R \otimes a^R$, where summation over more than one simple tensor is suppressed. In this case, the  multiplication in $B \otimes A$ is given by
\begin{equation}
(b_1 \otimes a_1)(b_2 \otimes a_2) = b_1 b_2^r \otimes a_1^r a_2
\end{equation}
 In order for $C$ to be associative, $R$ must satisfy two pentagonal commutative diagrams, equationally given by 
\begin{equation}
R( \mu_A \otimes B) = (B \otimes \mu_A)(R \otimes A)(A \otimes R)
\end{equation}  
(where $\mu_A$ denotes multiplication in $A$),  and
\begin{equation}
R(A \otimes \mu_B) = (\mu_B \otimes A)(B \otimes R)(R \otimes B)
\end{equation}
in $\Hom (A \otimes B \otimes B, B \otimes A)$. These equations are satisfied if and only if $C$ is associative.  
Additionally, the structure map $R$ satisfies two commutative triangles given equationally by $R(A \otimes 1_B) = 1_B \otimes A$ and $R(1_A \otimes B) = B \otimes 1_A$. It follows that $A \to C,
a \mapsto 1_B \otimes a$ and $B \to C, b \mapsto b \otimes 1_A$ are algebra monomorphisms. 
\begin{example}
\begin{rm}
Let $B$ be an algebra in ${}_H\M$, where $A = H$ is a Hopf algebra as before.  Let $R: B \otimes H \to H \otimes B$ be given by $R(b \otimes h) = h\1 . b \otimes h\2$.  Then $B \otimes^R  H =
B \# H$, the smash product of $H$ with a left $H$-module algebra $B$.  
\end{rm}
\end{example}
\begin{prop} \cite[Theorem 5.2]{HKY}
\label{th-young}
The minimum odd depth of $H$ embedded canonically in the smash product $B \# H$ satisfies
\begin{equation}
\label{eq: young}
d_{\rm odd}(H, B \# H) = d(B, {}_H\M) 
\end{equation} 
\end{prop}
\begin{proof}
Via cancellations of the type $X \otimes_H H \cong X$, one establishes an $H$-$H$-bimodule isomorphism, 
\begin{equation}
 (B \# H)^{\otimes_H n} \cong B^{\otimes (n)} \otimes H, 
\end{equation}
where the left $H$-module structure on $B^{\otimes (n)} \otimes H$ is given by the diagonal
action: $$x.(b_1 \otimes \cdots \otimes b_n \otimes h) = x\1. b_1  \otimes \cdots \otimes  x\n . b_n  \otimes  x_{(n+1)}. h$$
If 
$B^{\otimes (n+1)} \| q \cdot B^{\otimes (n)}$ in ${}_H \M$ for some $q \in \N$, then tensoring this by $- \otimes H$
yields 
$(B \# H)^{\otimes_H (n+1)} \| q \cdot (B \# H)^{\otimes_H n}$ as $H$-$H$-bimodules.
Thus the minimum odd depth  $d_{\rm odd}(H, B \# H) \leq d(B, {}_H\M)$ by Definition~\ref{def-depth}.  

Conversely, if $(B \# H)^{\otimes_H (n+1)} \| q \cdot (B \# H)^{\otimes_H n}$ as $H$-$H$-bimodules,
then $B^{\otimes (n+1)} \otimes H \| q \cdot B^{\otimes (n)} \otimes H$, to which one applies 
$ - \otimes {}_Hk$, obtaining $B^{\otimes (n+1)} \| q \cdot B^{\otimes (n)}$ in ${}_H \M$. Therefore $d(B,{}_H\M)  \leq d_{\rm odd}(H, B \# H)$.
\end{proof}
Using the notation developed in Section~3 for a finite-dimensional Hopf subalgebra pair $R \subseteq H$
with quotient right $H$-module coalgebra $Q$, we note that its $k$-dual $Q^*$ becomes a left $H$-module algebra via $\bra h q^*, q \ket = \bra q^* , qh \ket$.  Yet another equivalent formulation of the fundamental problem in Section~1 follows easily from the proposition since $d(Q^*, {}_H\M) = d(Q, \M_H)$ \cite{K2014, HKY}. 
\begin{problem}
Is $d(H, Q^* \# H) < \infty$ or $d(R, Q^* \# R) < \infty$? 
\end{problem}
\begin{example}
\label{ex-match}
\begin{rm}
Suppose $B$ and $H$ are  a \textit{matched pair} of Hopf algebras (see \cite[7.2.1]{Ma} or \cite[IX.2.2]{CK}). I.e., $H$ is a coalgebra in $\M_B$
with action denoted by $h \ract b$, and $B$ is coalgebra in ${}_H\M$ with action denoted by
$h \lact b$ satisfying compatibility conditions given in \cite[(7.7)-(7.9)]{Ma}.  A twisting $R: H \otimes B \to B \otimes H$ is given by 
\begin{equation}
R(h \otimes b) = h\1 \lact b\1 \otimes h\2 \ract b\2,
\end{equation}
which defines an algebra structure on $B \otimes^R H = B \bowtie H$; moreover, this is a Hopf algebra, called the double cross product, where $H$ and $B$ are canonically Hopf subalgebras \cite{Ma}.

For example, $H$ and its dual Hopf algebra (with opposite multiplication) $B = {H^{op}}^*$ are a matched pair via $\lact$, the left coadjoint action of $H$ on $H^*$, 
\begin{equation}
\label{eq: majid}
h \lact b = b\2 \bra (S b\1 ) b\3, h \ket, 
\end{equation}
and $\ract$ the analogous left coadjoint action of $H^*$ on $H$.  This defines the Drinfeld double $D(H)$ as a special case of double cross product,  $D(H) = {H^{op}}^*\bowtie H$.
\end{rm}
\end{example}

\begin{prop}
\label{prop-matched}
Let $B$ and $H$ be a matched pair of finite-dimensional Hopf algebras with $A = B \bowtie H$ their double cross product.
Then the minimum h-depth and even depth of the Hopf subalgebra $B$ in $A$ is given by the depth of $H$ in the finite tensor category $\M_B$ (w.r.t.\ $\ract$ in Example~\ref{ex-match}): $d_h(B,A) = d(H, \M_B)$ and 
 $d_{even}(B,A) = d(H, \M_B)+1$. Similarly, the Hopf subalgebra $H$ has depth in $A$
given by $d_h(H,A) = d(B, {}_H\M)$ (w.r.t.\ $\lact$) and $d_{even}(H,A) = d(B, {}_H\M) + 1$.  
\end{prop}
\begin{proof}
This follows from Cor.~\ref{cor-cue} if we show that the quotient module $Q_B \cong (H_B, \ract)$. 
Note that $Q = B \bowtie H / B^+ (B \bowtie H) \cong H$ via $\overline{b \bowtie h} \mapsto \eps_B(b)h$, and  
 $$\overline{h}b = \overline{(1_B \bowtie h)(b \bowtie 1_H)} = \overline{h\1 \lact b\1 \bowtie  h\1 \ract b\2} = \eps_B(h\1 \lact b\1) \overline{h\2 \ract b\2}$$
$ =  \overline{h \ract b}$,
where we use axiom~(\ref{eq: axiom}) for $B$, a left $H$-module coalgebra.
\end{proof}
For example, if $B = {H^{op}}^*$ and $B \bowtie H = D(H)$, suppose $H$ is cocommutative.  From the
formula for coadjoint action, it is apparent that $H_B \cong (\dim H) \cdot k$, so $d(H, \M_B) = 1$ and $d(H^*, D(H)) \leq 2$.  Indeed, it is known that $D(H) \cong H^* \# H$ in case $H$ is quasitriangular \cite[Majid, 1991, 7.4]{Ma}), but a smash product is a Hopf-Galois extension of its left $H$-module algebra
(which has depth $2$).  

\begin{example}
\label{ex-eight}
\begin{rm}
A study of the $8$-dimensional small quantum group $H_8$ (see for example \cite[Example 4.9]{K2014} for its Hopf algebra structure) and its quantum double $D(H_8)$
indicates that minimum depth satisfies $3 \leq d(H_8, D(H_8)) \leq 4$.  The method is to compute $D(H_8)$ in terms of generators and relations, compute the quotient $Q$ as an 
$8$-dimensional $H_8$-module, then decompose it into its indecomposable summands (twice each simple, and two $2$-dimensional indecomposables), compute the tensor products between these indecomposables, noting that $Q \sim Q \otimes Q$ as $H_8$-modules, and using Eq.~(\ref{eq: cue-even}). Since both algebras have infinite representation type, we cannot otherwise predict a finite depth from known results \cite{K2014, AHA}. 
\end{rm}
\end{example}  

Let ${}_{ad}H$ denote the adjoint action of $H$ on itself, given by $h . x = h\1 x S(h\2)$ for all
$h,x \in H$. 
\begin{cor}\cite[Cor.\ 5.4]{HKY}
\label{cor-group}
Let $G$ be a finite group and $D(G)$ its Drinfeld double as a complex group algebra. Then $d(\C G,D(G)) = d({}_{ad}\C G, {}_{\C G}\M)$. 
\end{cor}
\begin{proof}
From the remark about cocommutativity just above, the double $D(G) = H^* \# H$ (with $H = \C G$) is a smash product to which Proposition~\ref{th-young} applies:  thus $d_{odd}(\C G, D(G)) = d(H^*, {}_{\C G}\M)$. The smash product multiplication formula for $g, h \in G$, $p_g, p_h \in H^*$  one-point projections, is given by 
\begin{equation}
\label{eq: cocommutative}
( p_x \# g)(p_y \# h) = 
p_x p_{gyg^{-1}} \# gh
\end{equation}
which visibly demonstrates that ${}_HH^* \cong {}_{ad}H^* \cong {}_{ad}\C G$. 

 It remains to show that $d_{even}(\C G, D(G)) =  1 + d_{odd}(\C G, D(G))$.   Note that $S(p_x) = p_{x^{-1}}$, $$\cop^2(p_x) = \sum_{z,y \in G}
p_z \otimes p_{z^{-1}y} \otimes p_{y^{-1}x}$$
whence using Eq.~(\ref{eq: majid}) $$h \lact p_x = \sum_{z,y \in G} p_{z^{-1}y} \bra p_{z^{-1}} p_{y^{-1} x} , h \ket = \sum_{z,y \in G} \bra p_{z^{-1}}, h \ket \bra p_{y^{-1} x}, h \ket p_{z^{-1}y} = p_{hxh^{-1}},$$
the adjoint action of $h$ on $p_x$.  Use Proposition~\ref{prop-matched} to conclude the proof.  
\end{proof}


\section{Morita equivalent ring extensions}
In this section we continue a study of Morita equivalence of ring extensions in \cite{Mi,Ik, Y}, though with an emphasis on functors and categories.  We will briefly provide the classical background theory, and prove that depth, relative cyclic homology as well as the bipartite graphs of a semisimple complex subalgebra pair are all Morita invariant properties of a ring or algebra extension. In addition, we  note a natural example of Morita equivalence in towers of Frobenius extensions. 

Define \textit{two ring extensions $A \| B$ and $R \| S$ to be Morita equivalent} if there are additive equivalences  $\mathcal{P}: {}_R\M \to {}_A\M$  and $\mathcal{Q}: {}_S\M \to {}_B\M$ satisfying a commutative rectangle (up to a natural isomorphism) with respect to the functors of restriction from $R$-modules into $S$-modules,
and from $A$-modules into $B$-modules.
\begin{equation}
\label{def-cs}
\begin{diagram}
{}_R\M &&\rTo^{\sim}_{\mathcal{P}} && {}_A\M\\
\dTo^{\Res^R_S} && && \dTo_{\Res^A_B}\\
{}_S\M && \rTo^{\sim}_{\mathcal{Q}} && {}_B\M 
\end{diagram}
\end{equation}
The  requirement then is that there be a natural isomorphism $\mathcal{Q} \Res^R_S \stackrel{\sim}{\rightarrow} \Res^A_B \mathcal{P}$.     One shows in an exercise that this is an equivalence relation on ring extensions by using operations on natural transformation by functors.  

From ordinary Morita theory we know that $\mathcal{P}({}_RR) = {}_AP$, a progenerator such that $\End{}_AP \cong R$, so that $P$ is in fact an $A$-$R$-bimodule with $\mathcal{P}({}_RX) = P \otimes_R X$ for
all ${}_RX$.  The dual of P is unequivocally $P^* = \Hom (P_R, R_R)$, an $R$-$A$-bimodule, since
$\Hom ({}_AP, {}_AA) \cong P^*$ as $R$-$A$-bimodules by  \cite[Theorem 1.1]{M}.  Then $P^* \otimes_A -: {}_A\M \to {}_R\M$ is an inverse equivalence to $\mathcal{P}$:  one has bimodule isomorphisms $P^* \otimes_A P \cong {}_RR_R$
and $P \otimes_R P^* \cong {}_AA_A$.  

Similarly there is an invertible Morita bimodule ${}_BQ_S$, a left and right progenerator module, such
that $\mathcal{Q}({}_SY) = {}_BQ \otimes_S Y$.  The condition that the rectangle above commutes applied to $R \in {}_R\M$ becomes ${}_BQ \otimes_S R \cong {}_BP$, also valid as $B$-$R$-bimodules due to naturality, noted as an equivalent condition in the proposition below. 

\begin{example}
\label{ex-matrix}
\begin{rm}
Given a ring extension $R \supseteq S$, let $A = M_n(R) \supseteq B = M_n(S)$.  Of course,
$A$ and $R$ are Morita equivalent via $P = n \cdot R$, also $B$ and $S$ are Morita equivalent
via $Q = n \cdot S$. Note that $$ {}_BQ \otimes_S R_R \cong n \cdot R = {}_BP_R.$$  Thus, as one would expect, the ring extensions $R \supseteq S$ and $A \supseteq B$ are  Morita equivalent. 
\end{rm}
\end{example}

\begin{example}
\label{ex-isomo}
\begin{rm}
Suppose $B \subseteq A$ and $S \subseteq R$ are ring extensions with ring isomorphism $\psi: A \stackrel{\sim}{\rightarrow} R$ restricting to a ring isomorphism $\eta: B \stackrel{\sim}{\rightarrow} S$.  Defining bimodules ${}_AP_R := {}_{\psi}R_R$ and ${}_BQ_S := {}_{\eta}S_S$, one shows in an exercise that the two ring extensions are Morita equivalent. 
\end{rm}
\end{example}
The proposition below characterizes Morita equivalence of ring extensions in many equivalent ways, condition (2) being the definition in \cite{Mi,Ik,Y}.  
\begin{prop}
\label{prop-8items}
The following conditions on ring extensions $A \supseteq B$ and $R \supseteq S$ are equivalent:
\begin{enumerate}
\item $A \supseteq B$ and $R \supseteq S$ are Morita equivalent;
\item there are Morita bimodules ${}_AP_R$ and ${}_BQ_S$ satisfying ${}_BQ \otimes_S R_R \cong {}_BP_R$ \cite{Mi};
\item there are Morita bimodules ${}_AP_R$ and ${}_BQ_S$ satisfying ${}_RR \otimes_S {Q^*}_B \cong {}_R {P^*}_B$;
\item there are Morita bimodules ${}_AP_R$ and ${}_BQ_S$ satisfying ${}_AA \otimes_B Q_S \cong {}_AP_S$;
\item  there are Morita bimodules ${}_AP_R$ and ${}_BQ_S$  satisfying ${}_SQ^* \otimes_B A_A \cong {}_S {P^*}_A$;
\item the following rectangle, with sides representing the induction functors, commutes up to a natural isomorphism,
\begin{equation}
\label{eq: csind}
\begin{diagram}
{}_R\M &&\rTo^{\sim}_{\mathcal{P}} && {}_A\M\\
\uTo^{\Ind^R_S} && && \uTo_{\Ind^A_B}\\
{}_S\M && \rTo^{\sim}_{\mathcal{Q}} && {}_B\M 
\end{diagram}
\end{equation}
\item the following rectangle, with sides representing the coinduction functors, commutes up to a natural isomorphism,
\begin{equation}
\label{eq: cscoind}
\begin{diagram}
{}_R\M &&\rTo^{\sim}_{\mathcal{P}} && {}_A\M\\
\uTo^{\Coind^R_S} && && \uTo_{\Coind^A_B}\\
{}_S\M && \rTo^{\sim}_{\mathcal{Q}} && {}_B\M 
\end{diagram}
\end{equation}
\item any of the conditions above  stated identically with  right module categories $\M_R$, $\M_A$, $\M_S$, and $\M_B$
replacing the corresponding left module categories. 
\end{enumerate}
\end{prop}
\begin{proof}
(1) $\Rightarrow$ (2) is sketched above.  (2) $\Leftrightarrow$ (3) follows from the computation 
$$ {}_R{P^*}_B \cong {}_R\Hom (P_R, R_R)_B \cong  
 {}_R\Hom (Q \otimes_S R_R, R_R)_B  \cong {}_R\Hom (Q_S, R_S)_B   $$ 
$ \cong{}_R  R \otimes_S {Q^*}_B$ using adjoint theorems in \cite[pp. 240, 243]{AF}.  This shows
(2) $\Rightarrow$ (3). This argument reverses by using the reflexive property of progenerators (${}_A\Hom ({}_RP^*, {}_RR)_R \cong {}_AP_R$).  

(3) $\Rightarrow$ (4) and (8). The following rectangle is commutative up to a natural isomorphism:
$$\begin{diagram}
\M_R &&\rTo^{\sim}_{-\otimes_R P^*} && \M_A\\
\dTo^{\Res^R_S} && && \dTo_{\Res^A_B}\\
\M_S && \rTo^{\sim}_{-\otimes_S Q^*} && \M_B 
\end{diagram}$$
since for any module $X_R$ one has $$X \otimes_R {P^*}_B \cong X \otimes_R R \otimes_S {Q^*}_B \cong X \otimes_S {Q^*}_B.$$
To the natural isomorphism identifying the sides of this rectangle, apply the functor $-\otimes_B Q$ from
the left and the functor $- \otimes_A P$ from the right to obtain the following commutative rectangle up to natural isomorphism:
$$\begin{diagram}
\M_A &&\rTo^{\sim}_{-\otimes_A P} && \M_R\\
\dTo^{\Res^A_B} && && \dTo_{\Res^R_S}\\
\M_B && \rTo^{\sim}_{-\otimes_B Q} && \M_S 
\end{diagram}$$
(4) now follows from applying the rectangle to $A_A$. 
(4) $\Rightarrow$ (5).  The same type of argument as in (2) $\Rightarrow$ (3) above shows that $${}_S{P^*}_A \cong
{}_S\Hom ({}_AP, {}_AA)_A \cong {}_S\Hom ({}_BQ, {}_BB) \otimes_B A_A \cong {}_S Q^* \otimes_B A_A.$$ 

(4) $\Rightarrow$ (6).  By using  (4), compute for any module ${}_SY$,  $${}_AA \otimes_B Q \otimes_S Y \cong {}_AP \otimes_S Y \cong  {}_A P \otimes_R R \otimes_S Y,$$
which shows the rectangle (6) is commutative up to a natural isomorphism. The converse (6) $\Rightarrow$ (4) follows from applying the rectangle to ${}_SS \in {}_S\M$ as well as naturality.  

(5) $\Rightarrow$ (7) For any module ${}_SW$, it suffices to show that $P\otimes_R \Hom ({}_SR, {}_SW) \cong \Hom ({}_BA, {}_BQ \otimes_S W)$ using natural isomorphisms in \cite[20.6, 20.11, exercise 20.12]{AF} and (5):
$$  {}_AP\otimes_R \Hom ({}_SR, {}_SW) \cong {}_A\Hom ({}_SP^*, {}_SW) \cong {}_A\Hom ({}_SQ^* \otimes_B A, {}_SW) $$
$$ \cong {}_A\Hom ({}_BA, {}_B\Hom ({}_SQ^*, {}_SW)) \cong {}_A\Hom({}_BA, {}_BQ \otimes_S W)$$

The rest of the proof is similar and left as an exercise.
\end{proof}

In the following proposition, we note some  different, quick proofs for certain results in \cite{Ik}, while building up  results which show that depth and bipartite graphs are Morita invariants of ring extensions. 
\begin{prop}
\label{prop-equiv}
Suppose $A \| B$ and $R \| S$ are Morita equivalent ring extensions.  In the notation of the previous proposition, it follows that
\begin{enumerate}
\item if the extension $A \supseteq B$ is a separable, then $R \supseteq S$ is a separable extension \cite{Ik};
\item if the extension $A \supseteq B$ is QF, then $R \supseteq S$ is a QF extension \cite{Ik};
\item \label{item-questioned} if the extension $A \supseteq B$ is Frobenius, then $R \supseteq S$ is a Frobenius extension \cite{Ik};
\item if $B \subseteq A$ is a semisimple complex subalgebra pair, then so is $S \subseteq R$ with identical inclusion matrix and bipartite graph;
\item the following diagram of tensor categories and functors commutes up to natural isomorphism:
\begin{equation}
\label{eq: gee}
\begin{diagram}
{}_R\M_R &&\rTo^{\sim}_{F} && {}_A\M_A\\
\dTo^{\Res^{R^e}_{S^e}} && && \dTo_{\Res^{A^e}_{B^e}}\\
{}_S\M_S && \rTo^{\sim}_{G} && {}_B\M_B 
\end{diagram}
\end{equation}
where $F({}_RX_R) := {}_AP \otimes_R X \otimes_R {P^*}_A$ and $G({}_SY_S) := {}_BQ \otimes_S Y \otimes_S {Q^*}_B$ define tensor equivalences;
\item $G(R^{\otimes_S (n)}) \cong A^{\otimes_B (n)}$ as $B$-$B$-bimodules and $F(R^{\otimes_S (n)}) \cong A^{\otimes_B (n)}$ as $A$-$A$-bimodules for each $n \in \N$; 
\item the centralizers are isomorphic:  $A^B \cong R^S$ \cite{Ik};
\item the ring extensions $A \| B$ and $R \| S$ have the same minimum depth and h-depth. 
\end{enumerate}
\end{prop}
\begin{proof}
\begin{enumerate}
\item Let $0 \to V \to W \to U \to 0$ be a short exact sequence in ${}_A\M$ that is split exact when restricted to ${}_B\M$.  By Rafael's characterization \cite{Ra} of separability, the short exact sequence splits in ${}_A\M$.  The rest of the proof follows from applying the  commutative rectangle~(\ref{def-cs}).
\item Suppose ${}_AV$ is $(A,B)$-projective (or ``relative projective''), i.e., ${}_AV \| {}_AA \otimes_B V$ (or the multiplication epi $A \otimes_B V \to V$ splits as an $A$-module map).  By the relative Faith-Walker theorem for QF extensions \cite{BM}, $V$ is also $(A,B)$-injective:  i.e., the canonical $A$-module monomorphism $V \into \Hom ({}_BA, {}_BV)$ splits.  In fact the class of relative projectives coincides with the class of relative injectives for QF extensions.  It is clear from the commutative diagram~(\ref{eq: csind})
that the equivalence $\mathcal{P}$ sends relative projectives into relative projective; similarly, it is clear from
the commutative rectangle~(\ref{eq: cscoind}) that relative injectives are sent by an equivalence into relative injectives.  The rest of the proof is then an application of the relative Faith-Walker characterization of QF extension.  
\item The proof is an application of the commutative rectangles~(\ref{eq: csind}) and~(\ref{eq: cscoind}) and the characterization of Frobenius extensions as having naturally isomorphic induction and coinduction functors.  Suppose $R \supseteq S$ is Frobenius.  Then $$\Ind^A_B \mathcal{Q} \cong \mathcal{P}\Ind^R_S \cong \mathcal{P} \Coind^R_S \cong \Coind^A_B \mathcal{Q}.$$
Since $\mathcal{Q}$ is an equivalence, it follows that $\Ind^A_B$ and $\Coind^A_B$ are naturally isomorphic functors, whence $A \supseteq B$ is Frobenius.  
\item Let $V_1,\ldots,V_s$ be the simples of $S$ (up to isomorphism).  Then $U_i := Q \otimes_S V_i$ are representatives of the simple isoclasses of $B$ by Morita theory.  Induce each $V_i$ to an $R$-module, expressing this uniquely up to isomorphism as a sum of nonnegative multiples of the simples of $R$, $W_1,\ldots,W_r$:
$$ R \otimes_S V_i \cong \oplus_{j=1}^r r_{ij} W_j.$$
The $s \times r$ matrix is the inclusion matrix $K_0(S) \to K_0(R)$ of the semisimple complex subalgebra pair $S \subseteq R$. This matrix determines the bipartite graph of the inclusion, an edge connecting black dot $i$ with white dot $j$ in case the $(i,j)$-entry is nonzero.

Since $A$ and $R$ Morita equivalent rings, both are semisimple complex algebras; the same is true
of $B$ and $S$.  Moreover, their centers are isomorphic, thus $A$ and $R$ each have $r$ distinct simples, and $B$, $S$ each have $s$ pairwise nonisomorphic simples.  Denote the simples of $A$ by $X_1, \ldots, X_r$ where $X_i \cong P \otimes_R W_i$ for each $i$. Suppose the inclusion matrix of $B \subseteq A$ is given by $A \otimes_B U_i \cong
\oplus_{j=1}^r b_{ij} X_j$.  Since 
$$A \otimes_B U_i \cong A \otimes_B Q \otimes_S V_i \cong P \otimes_R R \otimes_S V_i \cong \oplus_{j=1}^r r_{ij} X_j$$
this implies by Krull-Schmidt that the inclusion matrices $(b_{ij})$ and $(r_{ij})$ are equal.  Thus the bipartite
graphs are equal.    
\item The functors $F$ and $G$ are  tensor equivalences according to Lemma~\ref{lem-tys}.  Let ${}_RX_R$ be a bimodule. Note that 
$ \Res^{A^e}_{B^e}(F(X)) = {}_BP \otimes_R X \otimes_R {P^*}_B$
$ \cong {}_B Q \otimes_S R \otimes_R X \otimes_R R \otimes_S {Q^*}_B \cong G(\Res^{R^e}_{S^e}(X))$
by applying (2) and (3) in Proposition~\ref{prop-8items}. Whence the rectangle is commutative. 
\item From the commutative rectangle just established it follows that $G({}_SR_S) \cong {}_BA_B$
and from the tensor functor property of $G$ that $G(R^{\otimes_S (n)} ) \cong {}_B{A^{\otimes_B (n)}}_B$.

A computation similar to the one in (4) of this proof shows that the following rectangle is commutative:
$$\begin{diagram}
{}_R\M_R &&\rTo^{\sim}_{F} && {}_A\M_A\\
\uTo^{\Ind^{R^e}_{S^e}} && && \uTo_{\Ind^{A^e}_{B^e}}\\
{}_S\M_S && \rTo^{\sim}_{G} && {}_B\M_B 
\end{diagram}$$
where $\Ind^{R^e}_{S^e}({}_SZ_S) := {}_RR \otimes_S Z \otimes_S R_R$. 
Since $F$ preserves tensor category unit objects, $F({}_RR_R) \cong {}_AA_A$. 
Starting with ${}_SS_S \in {}_S\M_S$, the rectangle shows that $F({}_RR \otimes_S R_R) \cong {}_AA \otimes_B A_A$.  Starting with $R^{\otimes_S (n)} \in {}_S\M_S$ in the rectangle, we note that for $n \geq 1$, 
$$F({}_R{R^{\otimes_S (n+2)}}_R) \cong \Ind^{A^e}_{B^e}(A^{\otimes_B (n)}) = {}_A{A^{\otimes_B (n+2)}}_A.$$
\item Note the equivalence of bimodule categories $H: {}_S\M_R \to {}_B\M_A$ given by $H({}_SW_R) := {}_BQ \otimes_S W \otimes _R {P^*}_A$. We claim that $H({}_SR_R) \cong {}_BA_A$; moreover,
\begin{equation}
\label{eq: steven}
H({}_S{R^{\otimes_S (n)}}_R) \cong {}_B{A^{\otimes_B (n)}}_A
\end{equation}
 for all $n \geq 1$.  
This follows from the diagram below, commutative up to natural isomorphism.  
\begin{equation}
\label{fig-mixedmodules}
\begin{diagram}
{}_R\M_R &&\rTo^{\sim}_{F} && {}_A\M_A\\
\dTo^{\Res^R_S} && && \dTo_{\Res^A_B}\\
{}_S\M_R && \rTo^{\sim}_{H} && {}_B\M_A
\end{diagram}
\end{equation}
which is established by a short computation  using (2) in Prop.~\ref{prop-8items}. Applied to $R^{\otimes_S (n)} \in {}_R\M_R$, we obtain Eq.~(\ref{eq: steven}). 

Note that the centralizer $R^S = \{ r \in R \, : \, \forall s \in S, rs = sr \}$ is isomorphic to  $\End ({}_SR_R) \cong R^S $ via $f \mapsto f(1)$. Recall that an equivalence $H$ satisfies $$\End ({}_SR_R)  \cong \End (H({}_S R_R)) \cong \End ({}_BA_A) \cong A^B. $$
\item Similarly to Eq.~(\ref{eq: steven}), we establish that the equivalence of bimodule categories
given by $H': {}_R\M_S \to {}_A\M_B$, ${}_RV_S \mapsto P \otimes_R V \otimes_S Q^*$ satisfies
\begin{equation}
\label{eq: steven'}
H'(R^{\otimes_S (n)}) \cong {}_A{A^{\otimes_B (n)}}_B
\end{equation}
Of course, equivalences preserve similarity of modules since they are additive.  Suppose
$R^{\otimes_S (n)} \sim R^{\otimes_S (n+1)}$ as $R$-$S$-bimodules, i.e., $R \| S$ has right depth $2n$.  Applying $H'$, one obtains $A^{\otimes_B (n)} \sim A^{\otimes_B (n+1)}$ as $A$-$B$-bimodules, i.e., $A \| B$ has right depth $2n$.  Similarly for left depth $2n$ using the equivalence $H$.  Similarly, if $R \| S$ has depth $2n+1$, applying $G$ we obtain that $A \| B$ has depth $2n+1$.  Going in the reverse direction using $G^{-1}$, $H^{-1}$, we obtain $d(S,R) = d(B,A)$.  
Using $F$ we likewise show that $d_h(S,R) = d_h(B,A)$. 
\end{enumerate}
\end{proof}

\subsection{Example: tower above Frobenius extension} \label{subsect-tower}
A Frobenius extension $A \supseteq B$ is characterized by any of the following four conditions
\cite{K}.
First, that $A_B$ is finite projective and ${}_BA_A \cong \Hom (A_B, B_B)$. Secondly,
that ${}_BA$ is finite projective and ${}_AA_B \cong \Hom ({}_BA, {}_BB)$.  Thirdly,
that coinduction and induction of right (or left) $B$-modules into $A$-modules are naturally isomorphic functors.  Fourth,
there is a Frobenius coordinate system $(E: A \to B; x_1,\ldots,x_m,$ $y_1,\ldots,y_m \in A)$, which satisfies $(\forall a \in A)$
\begin{equation}
\label{eq: FEQ's}
E \in \Hom ({}_BA_B, {}_BB_B), \ \ \sum_{i=1}^m E(ax_i)y_i = a = \sum_{i=1}^m x_i E(y_i a).
\end{equation}
These  equations may be used to show that $\sum_i x_i \o y_i \in (A \o_B A)^A$.  

By \cite[Lemma 4.1]{LK2012}, a Frobenius extension $A \supseteq B$ has both $A_B$ and ${}_BA$  generator modules if and only if the Frobenius homomorphism $E: A \to B$ is surjective:  although most Frobenius extensions in the literature
are generator extensions, there is a somewhat pathological example  in \cite[2.7]{K}  of a matrix algebra Frobenius extension with a non-surjective Frobenius homomorphism.

A Frobenius  extension $A \supseteq B$ enjoys an endomorphism ring theorem, which states that $ A_2 := \End A_B \supseteq A$ is itself a Frobenius  extension,  where the ring monomorphism $A \to A_2$ is 
the left multiplication mapping $\lambda: a \mapsto \lambda_a$ , $\lambda_a(x) = ax$.   
It is worth noting that $\lambda$ is a left split $A$-monomorphism (by evaluation at $1_A$) so ${}_AA_2$ is a generator.  It is an exercise to check that $A_2 \cong A \otimes_B A$ via
$f \mapsto \sum_i f(x_i) \otimes_B y_i$; the induced ring structure on $A \otimes_B A$ is the
``E-multiplication,'' given by 
\begin{equation}
(a \otimes_B c)(d \otimes_B e) = aE(cd) \otimes_B e.
\end{equation}
The identity is given $1 = \sum_i x_i \otimes_B y_i$.  The Frobenius coordinate system for
$A_2 \supseteq A_1$ is given by $E_2(a \otimes_B c) = ac$ (always surjective!) with dual bases $\{ x_i \otimes_B 1 \} $
and $\{ 1 \otimes_B y_i \}$.  

The \textit{tower} of a Frobenius extension is obtained by iteration of the endomorphism ring and $\lambda$,
obtaining a tower of Frobenius extensions; with the notation $B := A_0, A := A_1$ and defining
$A_{n+1} = \End {A_n}_{A_{n-1}}$, we obtain the tower, 
\begin{equation}
\label{eq: tower}
A_0 \into A_1 \into A_2 \into \cdots \into A_n \into A_{n+1} \into \cdots
\end{equation}
By transitivity of Frobenius extension or QF extension \cite{P}, all sub-extensions $A_m \into A_{m+n}$ in the tower are also Frobenius  extensions. Note that $A_n \cong A^{\otimes_B (n)}$: the ring, module and Frobenius structures in the tower are worked out in \cite{LK2012}. 

\begin{theorem}
Suppose $A \supseteq B$ is a Frobenius extension with the tower and data notation given above.  Then $A_{n-1} \supseteq A_{n-2}$
is Morita equivalent to $A_{n+1} \supseteq A_n$ for all integers $n > 1$.  Also $A \supseteq B$ is Morita equivalent
to $A_3 \supseteq A_2$ if the Frobenius homomorphism is epi.
\end{theorem}
\begin{proof}
It suffices to assume $E: A \to B$ is surjective, let $S = A_2 = \End A_B$, $R = A_3$, and show that $B \into A$ is Morita equivalent to $A_2 \into A_3$. Since $A$ is a Frobenius extension of $B$ with surjective Frobenius homorphism, it follows that the module  $A_B$ is a progenerator; since $A_2 = \End A_B$, it follows that $B$ and $A_2$ are Morita equivalent rings.  Similarly, $A$ and $A_3 \cong \End A \otimes_B A_A$
are Morita equivalent rings. 

In the notation of Proposition~\ref{prop-8items} (exchanging $R$ with $A$, $B$ with $S$), note that $Q = A$ and $P = A \otimes_B A$. Thus ${}_SQ \otimes_B A_A \cong {}_SP_A$, the condition
in the proposition for Morita equivalent ring extensions.  
\end{proof}
The theorem states in other words that the tower above a Frobenius extension has up to Morita equivalence period two. 
Note that consecutive ring extensions in the tower are almost never Morita equivalent:  in \cite[Example 1.12]{LK2012},
the depth is $d(S_3,S_4) = 5$, but of its reflected graph, the depth is $d(A,A_2) = 6$  (where $A = \C S_4$, using the graph-theoretic depth calculation in \cite[Section~3]{BKK}). 
 
\subsection{Relative cyclic homology of ring extensions is  Morita invariant}  We extend a result in \cite{LK92} that relative cyclic homology of a ring extension $R \supseteq S$ and of its $n \times n$-matrix ring extension $M_n(R) \supseteq M_n(S)$ are isomorphic via a Dennis trace map adapted to this set-up.
The relative cyclic homology (or any of its several variant homologies) is computed from cyclic modules
$$Z_n(R,S) := R \otimes_{S^e} R^{\otimes_S (n)},$$ which has the effect of considering tensor products
of the natural bimodule ${}_SR_S$ with itself over $S$ $n+1$ times arranged in a circle (in place of a line). For each $n \geq 0$, there are $n+1$ face maps are given by $d_i: Z_n(R,S) \rightarrow Z_{n-1}(R,S)$ defined from tensoring $n-1$ copies of  the $\id_{{}_SR_S}$ with one copy of the multiplication $\mu \in \Hom ({}_SR \otimes_S R_S, {}_SR_S)$ at the $i$´th position,
there are $n+1$ degeneracy mappings $s_j:  Z_n(R,S) \to Z_{n+1}(R,S)$ by tensoring $n$ copies of $\id_{{}_SR_S}$ with one copy of the unit mapping $\eta \in \Hom ({}_SS_S, {}_SR_S)$ in the $i$´th position, and a cyclic permutation 
$t_n: Z_n(R,S) \to Z_n(R,S)$ of order $n+1$ (see \cite{LK92} for the Connes cyclic object relations \cite{C} and the textbook \cite{Lo} for further details).

Suppose ring extensions $R \supseteq S$ and $A \supseteq B$ are Morita equivalent, and assume the same
structural bimodules and module equivalences with notation as in this section.  
Now recall from the diagram~(\ref{eq: gee}) that the tensor equivalence $G: {}_S\M_S \to {}_B\M_B$,
defined by $G(X) = Q \otimes_S X \otimes_S Q^*$,
sends ${}_SR_S$ into ${}_BA_B$.  We note the following commutative diagram,

\begin{equation}
\label{fig-bifunctors}
\begin{diagram}
{}_S\M_S \times  {}_S\M_S &&\rTo^{\sim}_{G \times G} && {}_B\M_B \times  {}_B\M_B \\
\dTo^{- \otimes_{S^e} -} && && \dTo_{- \otimes_{B^e}- }\\
Ab_S && \rTo^{\sim}_{\hat{G}} && Ab_B
\end{diagram}
\end{equation}   
\newline
where $Ab_B$ denotes ${}_B\M_B \otimes_{B^e} {}_B\M_B$, a subcategory of abelian groups  (and
similarly for $Ab_S$), 
from a computation with $X,Y \in {}_S\M_S$:
$$ G(X) \otimes_{B^e} G(Y) \cong X \otimes_S Q^* \otimes_B Q \otimes_{S^e} Q^* \otimes_B Q \otimes_S Y $$
$$ \cong X \otimes_S S  \otimes_{S^e}  S \otimes_S Y \cong X  \otimes_{S^e}  Y. $$
It follows that $Z_n(R,S) \stackrel{\cong}{\longrightarrow} Z_n(A,B)$ via $\hat{G}$ (restricted to the cyclic modules) as abelian groups for each $n \geq 0$. 
Now  $\hat{G}$   commutes
with face maps since the functor $G$ sends   the multiplication of $R \supseteq S$, 
$$ \mu \in \Hom ({}_SR \otimes_S R_S, {}_SR_S) \mapsto \mu´ \in \Hom ({}_BA \otimes_B A_B, {}_BA_B),$$
the multiplication of the ring extension $A \supseteq B$.  That $\hat{G}: Z_n(R,S) \to Z_n(A,B)$ commutes with the degeneracy maps follows
from the functor $G$ sending the unit $\eta \in \Hom ({}_SS_S, {}_SR_S)$ into the unit $\eta´ \in \Hom ({}_BB_B, {}_BA_B)$.  That $\hat{G}: Z_n(R,S) \to Z_n(A,B)$ commutes with the cyclic group action generator $t_n$ follows from $G \times G$ commuting with simple exchange $X \times Y \mapsto Y \times X$.  We have sketched the proof of the next proposition.

\begin{prop}
If $R \supseteq S$ and $A \supseteq B$ are Morita equivalent ring extensions, then their cyclic modules, cyclic chain complexes and cyclic homology groups 
are isomorphic:  $HC_n(R,S) \cong HC_n(A,B)$, all $n \in \N$.
\end{prop}
The isomorphism is given by a generalized Dennis trace mapping as follows.  Suppose the $S$-bimodule isomorphism $Q^* \otimes_B Q \stackrel{\cong}{\longrightarrow} S$ sends $\sum_{i=1}^r q_i^* \otimes q_i \mapsto 1_S$.  Then an isomorphism of cyclic modules $Z_n(A,B)\stackrel{\cong}{\longrightarrow} Z_n(R,S) $ is given by
\begin{equation}
\label{eq: dennis}
a_0 \otimes \cdots \otimes a_n \mapsto \sum_{i_0,\ldots,i_n=1}^r q_{i_0} \otimes a_0 \otimes q_{i_1}^* \otimes q_{i_1} \otimes \cdots \otimes q_{i_n} \otimes a_n \otimes q_{i_0}^*
\end{equation}
In the matrix example~\ref{ex-matrix} of Morita equivalent ring extensions, where each $a_i$ denotes an $n\times n$-matrix, this expression simplifies to the classical Dennis trace isomorphism of cyclic modules  noted in \cite{LK92},
$$a_0 \otimes_{B^e}a_1 \otimes_B \cdots \otimes_B a_n \mapsto \sum_{i_0, \ldots,i_n=1}^r a_0^{i_0 i_1} \otimes_{S^e} a_1^{i_1 i_2} \otimes_S \cdots \otimes_S a_n^{i_n i_0}. $$

\subsection{Acknowledgements}   The author thanks the organizers of the ``New Trends in Hopf algebras and Tensor Categories'' in Brussels, June 2-5, 2015, for a nice conference including Joost Vercruysse and Mio Iovanov for discussions about Proposition~\ref{prop-equiv}, item (\ref{item-questioned}), Alberto Hernandez for interesting mathematical conversations about several subjects in this paper, and  
CMUP (UID/MAT/00144/2013), which is funded by FCT (Portugal) with national (MEC) and European structural funds through the programs FEDER, under the partnership agreement PT2020, as well as  Professor Manuel Delgado, and the Invited Scientist program of CMUP  for financial support.

\end{document}